\documentclass[letterpaper, pdftex,11pt]{amsart} 
\usepackage[pdftex]{hyperref} 
\usepackage{amsmath,amssymb}
\usepackage{mathtools}
\usepackage{amscd}
\usepackage{mathrsfs}
\usepackage[foot]{amsaddr}
\usepackage{amsthm}
\usepackage{setspace}
\usepackage[all]{xy}
\usepackage{comment}
\usepackage{graphicx} 
\usepackage{tikz}
\usetikzlibrary{arrows}
\usetikzlibrary{shapes.geometric,fit}
\usepackage{tikz-cd}
\usepackage{fancyhdr}

\theoremstyle{definition}
\newtheorem{thm}{Theorem}[subsection]
\newtheorem*{thm*}{Theorem}
\newtheorem{defi}[thm]{Definition}
\newtheorem*{defi*}{Definition}
\newtheorem*{acknowledge}{Acknowledgement}
\newtheorem{cor}[thm]{Corollary}
\newtheorem*{cor*}{Corollary}
\newtheorem{prop}[thm]{Proposition}
\newtheorem*{prop*}{Proposition}
\newtheorem{lem}[thm]{Lemma}
\newtheorem*{lem*}{Lemma}
\newtheorem{ex}[thm]{Example}
\newtheorem*{ex*}{Example}

\newtheorem{rem}[thm]{Remark}
\newtheorem*{rem*}{Remark}

\newtheorem*{claim*}{Claim}

\newtheorem*{hw*}{Homework}
\newtheorem{thmintro}{Theorem}[section]

\newcommand{\C}{\mathbb{C}}

\newcommand{\Z}{\mathbb{Z}}
\newcommand{\N}{\mathbb{N}}
\newcommand{\T}{\mathbb{T}}

\DeclareMathOperator{\supp}{supp}
\DeclareMathOperator{\Tw}{Tw}

\newcommand{\Bis}{\mathrm{Bis}}
\DeclareMathOperator{\Iso}{\mathrm{Iso}}
\DeclareMathOperator{\Hom}{\mathrm{Hom}}

\DeclareMathOperator{\Span}{\mathrm{span}}

\DeclareMathOperator{\dom}{\mathrm{dom}}

\DeclareMathOperator{\id}{\mathrm{id}}
\DeclareMathOperator{\Aut}{\mathrm{Aut}}

\DeclareMathOperator{\insupp}{\mathrm{supp}^{\circ}}
\def\i<#1>{\langle #1 \rangle}
\def\l<#1>{\left\langle #1 \right\rangle}

\makeatletter
\renewenvironment{proof}[1][\proofname]{\par
  \normalfont
  \topsep6\p@\@plus6\p@ \trivlist
  \item[\hskip\labelsep{\bfseries #1}\@addpunct{.}]\ignorespaces
}{
  \endtrivlist
}
\renewcommand{\proofname}{\sc{Proof}}
\makeatother
\makeatletter
\newcommand*{\defeq}{\mathrel{\rlap{%
                     \raisebox{0.3ex}{$\m@th\cdot$}}%
                     \raisebox{-0.3ex}{$\m@th\cdot$}}%
                     =}
\makeatother

\makeatletter
\@namedef{subjclassname@2020}{\textup{2020} Mathematics Subject Classification} 
\makeatother

\title[]{Cartan-preserving *-automorphism groups: realization and obstructions for compact abelian groups}
\author{Fuyuta Komura}
\address{Kyushu institute of technology}
\email{komura@mns.kyutech.ac.jp}

\begin{document}
\maketitle
\begin{abstract}
	
	In this paper,
	we investigate automorphism groups preserving Cartan subalgebras of C*-algebras.
	First,
	we describe these groups in terms of automorphisms and 1-cocycles of twisted étale groupoids.
	As a consequence,
	we obtain a C*-algebraic analogue of a theorem of Feldman and Moore on Cartan-preserving automorphisms of von Neumann algebras.
	We then study Cartan-fixing automorphism groups.
	We show that every UCT Kirchberg algebra admits a Cartan subalgebra whose Cartan-fixing automorphism group contains every second countable compact abelian group.
	In contrast,
	for C*-algebras arising from expansive effective groupoids,
	we prove that compact Cartan-fixing automorphism groups must have finitely generated Pontryagin duals.
	As an application,
	we establish the existence of inequivalent Cartan subalgebras for Kirchberg algebras arising from expansive effective groupoids.
\end{abstract}

	\section{Introduction}
	
	Cartan subalgebras play a central role in the study of operator algebras.
	A celebrated theorem of Feldman and Moore of \cite{Feldman_Moore_II} asserts that every von Neumann algebra with a Cartan subalgebra arises from a Borel equivalence relation together with a 2-cocycle.
	In addition,
	they studied the groups of Cartan-preserving automorphisms in terms of Borel equivalence relations and 2-cocycles.
	In the C*-algebraic setting,
	Renault proved that every separable C*-algebra with a Cartan subalgebra arises from a twist over an effective \'etale groupoid in \cite{renault}.
	This result was later generalized to the non-separable setting by Raad in \cite{Ali_Raad_generalization_of_Renault}.
	These results provide a dynamical model for C*-algebras admitting Cartan
	subalgebras in terms of twisted \'etale groupoids.
	This viewpoint suggests that structural properties of C*-algebras admitting Cartan subalgebras should be reflected in the corresponding twisted groupoids.
	This philosophy has proved fruitful in the study of groupoid C*-algebras.
	For example,
	several authors have successfully described the ideal structure of twisted groupoid C*-algebras in terms of the underlying groupoids; see
	\cite{Armstrong_uniqueness_theorem_for_twistedgroupoid,
		Armstrong_Brownlowe_Sims_simplicity_of_twistedDeaconu,
		KangLiidealstr,
		Brix_Carsen_Sims2024,
		Brown2014}.
	Moreover, pure infiniteness of groupoid C*-algebras can also be characterized through underlying groupoids; see \cite{AnatharamanDelapurelyinfinite, Brown_Clark_Sierakowski_purelyinfinite}.
	On the other hand, Cartan subalgebras are not unique in general.
	Indeed, many C*-algebras admit several inequivalent Cartan subalgebras (see \cite{XinLiRenaultCartanexistence}, for example).
	Consequently,
	it is important to investigate intrinsic properties of Cartan inclusions and to develop invariants capable of distinguishing them.
	Recent developments in this direction include the study of diagonal dimension \cite{li2023diagonaldimensionsubcalgebras},
	which may be viewed as a dimension theory for inclusions of C*-diagonals.
	
	In the present paper,
	we investigate Cartan-preserving automorphism groups in order to understand the structural properties of Cartan inclusions.
	Let $G$ be a locally compact Hausdorff \'etale groupoid and $\Sigma$ be a twist over $G$.
	Then one associates an inclusion of C*-algebras $C_0(G^{(0)})\subset C^*_r(\Sigma; G)$.
	If $G$ is effective,
	then $C_0(G^{(0)})\subset C^*_r(\Sigma;G)$ is a Cartan subalgebra.
	Conversely,
	Renault's reconstruction theorem \cite{renault} asserts that every Cartan inclusion of a separable C*-algebra arises in this way.
	For an inclusion of C*-algebras $D\subset A$,
	we put
	\begin{align*}
		&\Aut(A;D)\defeq \{\varphi\in \Aut(A)\mid \varphi(D)=D\},\\
		&\Aut_D(A)\defeq\{\varphi\in \Aut(A)\mid \varphi(d)=d\text{ for all $d\in D$}\}.
	\end{align*}
	The main objects in this paper are the following two groups:
	\begin{align*}
		\Aut(C^*_r(\Sigma;G); C_0(G^{(0)})), \Aut_{C_0(G^{(0)})}(C^*_r(\Sigma; G)).
	\end{align*}
	The first group consists of automorphisms preserving the Cartan subalgebra,
	while the second consists of automorphisms fixing it pointwise.
	We remark that the second group is a normal subgroup of the first.
	Taylor's theorem \cite[Theorem 6.10]{Jonarhan2025essentialCartan} shows that $\Aut(C^*_r(\Sigma;G); C_0(G^{(0)}))$ is isomorphic to the automorphism group $\Aut(\Sigma;G)$ of the twisted groupoid $(\Sigma, G)$ when $G$ is effective.
	More generally,
	Taylor established this result for effective \'etale groupoids that are not necessarily Hausdorff in \cite{Jonarhan2025essentialCartan}.
	While Taylor's theorem provides a complete description of $\Aut(C^*_r(\Sigma;G);C_0(G^{(0)}))$,
	the structure of its subgroup $\Aut_{C_0(G^{(0)})}(C^*_r(\Sigma;G))$ has remained unclear.
	In the present paper,
	we begin by investigating the inclusion 
	\[
	\Aut_{C_0(G^{(0)})}(C^*_r(\Sigma; G))\subset \Aut(C^*_r(\Sigma;G); C_0(G^{(0)})).
	\]
	More precisely,
	our first main theorem is the following:
	
	\begin{thmintro}[{Corollary \ref{cor: C*-analogue of Feldman moore}}]\label{thmintro: C*-analogue of Feldman}
	Let $G$ be a locally compact Hausdorff effective \'etale groupoid and $\Sigma$ be a twist over $G$.
	Put
	\[
	\mathfrak{W}_{C^*_r(\Sigma; G), C_0(G^{(0)})}\defeq \Aut(C^*_r(\Sigma; G); C_0(G^{(0)}))/\Aut_{C_0(G^{(0)})}(C^*_r(\Sigma;G)).
	\]
	Then the short exact sequence
	\[
	\Aut_{C_0(G^{(0)})}(C^*_r(\Sigma;G))\to \Aut(C^*_r(\Sigma;G); C_0(G^{(0)}))\to \mathfrak{W}_{C^*_r(\Sigma; G), C_0(G^{(0)})}
	\]
	is isomorphic to
	\[
	Z(G)\to \Aut(\Sigma;G)\to \Aut(G)_{\Sigma},
	\]
	where $Z(G)$ is the abelian group of 1-cocycles on $G$ and $\Aut(G)_{\Sigma}$ is the stabilizer group of the action $\Aut(G)\curvearrowright \Tw(G)$ at $\Sigma$ (see Proposition \ref{prop: left action of AutG on TwG} for this action).
	\end{thmintro}
	
	We remark that Theorem \ref{thmintro: C*-analogue of Feldman} can be viewed as a C*-algebraic analogue of \cite[Theorem 3]{Feldman_Moore_II},
	where the authors proved a corresponding statement in the von Neumann algebraic setting.
	An important distinction from the von Neumann algebraic setting is that a
	twist over an \'etale groupoid is not necessarily represented by a groupoid
	$2$-cocycle.
	Therefore, the role played by the cohomology class of a $2$-cocycle in the work of Feldman and Moore must be replaced by that of a twist.
	Part of this paper is devoted to developing the necessary cohomological framework for twists.
	
	By Theorem \ref{thmintro: C*-analogue of Feldman},
	we completely determine the structure of the short exact sequence
	\[
	\Aut_{C_0(G^{(0)})}(C^*_r(\Sigma;G))\to \Aut(C^*_r(\Sigma;G); C_0(G^{(0)}))\to \mathfrak{W}_{C^*_r(\Sigma; G), C_0(G^{(0)})}
	\]
	in terms of the underlying twisted groupoid.
	In particular,
	it follows from Theorem \ref{thmintro: C*-analogue of Feldman} that
	\[
	\Aut_{C_0(G^{(0)})}(C^*_r(\Sigma;G))\simeq Z(G).
	\]
	Thus,
	the group of Cartan-fixing automorphisms $\Aut_{C_0(G^{(0)})}(C^*_r(\Sigma;G))$ is an abelian group and depends only on the underlying \'etale groupoid $G$,
	but not on the twist $\Sigma$.
	Using this description,
	we investigate which abelian groups can arise as Cartan-fixing automorphism groups.
	Our next main theorem (Theorem \ref{thmintro:realization theorem}) shows that,
	for every Kirchberg algebra satisfying the UCT,
	one can choose a Cartan subalgebra whose Cartan-fixing automorphism group contains every second countable compact abelian group.
	Moreover,
	this realization can be arranged so that the corresponding fixed point algebras remain Kirchberg algebras satisfying the UCT.
	
	\begin{thmintro}[{Theorem \ref{Theorem: existence of universal Cartan subalgebra of Kirchberg algebra}}]\label{thmintro:realization theorem}
		Let $A$ be a unital Kirchberg algebra satisfying the UCT.
		Then there exists a Cartan subalgebra $D\subset A$ whose Gelfand spectrum is a Cantor space with the following property:
		for any second countable compact abelian group $H$,
		there exists a continuous injective homomorphism $\tau\colon H\hookrightarrow  \Aut_D(A)$ such that the following conditions hold:		
		\begin{enumerate}
			\item the fixed point algebra $A^{\tau}$ is a unital Kirchberg algebra satisfying the UCT,
			\item $D\subset A^{\tau}$,
			\item $\tau(H)=\Aut_{A^\tau}(A)$.
		\end{enumerate} 
		In particular,
		$H$ is isomorphic to $\Aut_{A^\tau}(A)$ via $\tau$.
	\end{thmintro}
	
	Our next main theorem shows that this phenomenon does not occur for arbitrary Cartan subalgebras.
	
	\begin{thmintro}[{Theorem \ref{theorem every compact subgroup of Cartan fixing are fin gen}}]
		Let $G$ be a locally compact Hausdorff expansive effective \'etale groupoid.
		Assume that $H\subset \Aut_{C_0(G^{(0)})}(C^*_r(G))$ is compact and the fixed point subalgebra $C^*_r(G)^H$ is prime.
		Then $\widehat{H}$ is finitely generated.
	\end{thmintro}
	
	In the above situation,
	we also prove that the Pontryagin dual of $\Aut_B(C^*_r(G))$ is finitely generated for a prime intermediate subalgebra $C_0(G^{(0)})\subset B\subset C^*_r(G)$ whenever  $\Aut_B(C^*_r(G))$ is compact (Corollary \ref{cor: Gal group of prime subalgebra is fin gen}).
	In contrast to Theorem \ref{thmintro:realization theorem},
	these results show that expansiveness imposes a strong finiteness condition on
	Cartan-fixing automorphism groups.
	
	As an application,
	we use Cartan-fixing automorphism groups to distinguish Cartan subalgebras and thereby obtain the existence of inequivalent Cartan subalgebras (Corollary \ref{cor: inequivalent Cartan}).
	We remark that the non-uniqueness of Cartan subalgebras itself is not new (see \cite[Proposition 5.7]{XinLiRenaultCartanexistence}).
	However,
	the techniques used in their proof differ from those employed in ours.
	In \cite[Proposition 5.7]{XinLiRenaultCartanexistence},
	the authors distinguished Cartan subalgebras by examining the ranks of isotropy groups arising from their groupoid models.
	In contrast, we distinguish Cartan subalgebras via the automorphism groups that fix them.

	This paper is organized as follows.
	In Section 1,
	we recall basic notions and preliminaries concerning twisted groupoids and their associated C*-algebras.
	In Section 2,
	we investigate automorphism groups of twisted groupoids.
	Then we establish a C*-algebraic analogue of the Feldman--Moore theorem (Corollary \ref{cor: C*-analogue of Feldman moore}).
	In Section 3,
	we investigate realization and obstruction problems for Cartan-fixing automorphism groups.
	Then we show the existence of inequivalent Cartan subalgebras.
	
	\begin{acknowledge}
		The author would like to thank Yuki Arano for fruitful discussions.
		This work was supported by JSPS KAKENHI Grant Number JP26K17000.
	\end{acknowledge}
	\section{Preliminaries}
	
	In this section,
	we recall fundamental facts about twisted groupoid C*-algebras.
	We refer the reader to \cite{asims}, \cite{paterson2012groupoids}, \cite{renault1980groupoid} and \cite{renault} for further details.
	
	\subsection{\'Etale groupoids}
	
	We recall the basic notions concerning \'etale groupoids.
	See \cite{asims} and \cite{paterson2012groupoids} for more details.
	
	A groupoid is a set $G$ together with a distinguished subset $G^{(0)}\subset G$,
	domain and range maps $d,r\colon G\to G^{(0)}$ and a product 
	\[
	G^{(2)}\defeq \{(\alpha,\beta)\in G\times G\mid d(\alpha)=r(\beta)\}\ni (\alpha,\beta)\mapsto \alpha\beta \in G
	\]
	such that
	\begin{enumerate}
		\item for all $x\in G^{(0)}$, $d(x)=x$ and $r(x)=x$ hold,
		\item for all $\alpha\in G$, $\alpha d(\alpha)=r(\alpha)\alpha=\alpha$ holds,
		\item for all $(\alpha,\beta)\in G^{(2)}$, $d(\alpha\beta)=d(\beta)$ and $r(\alpha\beta)=r(\alpha)$ hold,
		\item if $(\alpha,\beta),(\beta,\gamma)\in G^{(2)}$,
		we have $(\alpha\beta)\gamma=\alpha(\beta\gamma)$,
		\item\label{inverse} every $\gamma \in G$,
		there exists $\gamma'\in G$ which satisfies $(\gamma',\gamma), (\gamma,\gamma')\in G^{(2)}$ and $d(\gamma)=\gamma'\gamma$ and $r(\gamma)=\gamma\gamma'$.   
	\end{enumerate}
	Since the element $\gamma'$ in (\ref{inverse}) is uniquely determined by $\gamma$,
	$\gamma'$ is called the inverse of $\gamma$ and denoted by $\gamma^{-1}$.
	We call $G^{(0)}$ the unit space of $G$.
	A subgroupoid of $G$ is a subset of $G$ which is closed under the inversion and product. 
	For $U\subset G^{(0)}$, we define $G_U\defeq d^{-1}(U)$ and $G^{U}\defeq r^{-1}(U)$.
	We define also $G_x\defeq G_{\{x\}}$ and $G^x\defeq G^{\{x\}}$ for $x\in G^{(0)}$.
	A subset $F\subset G^{(0)}$ is said to be invariant if $d(\alpha)\in F$ implies $r(\alpha)\in F$ for all $\alpha\in G$.
	If $F\subset G^{(0)}$ is invariant,
	$G_F\subset G$ is a subgroupoid and the unit space of $G_F$ is $F$.
	
	A topological groupoid is a groupoid equipped with a topology where the product and the inverse are continuous.
	A topological groupoid is said to be \'etale if the domain map is a local homeomorphism.
	Note that the range map of an \'etale groupoid is also a local homeomorphism.
	In this paper,
	we assume that all topological groupoids are Hausdorff unless otherwise stated,
	although important examples of non-Hausdorff groupoids exist.
			
	A subset $U$ of an \'etale groupoid $G$ is called a bisection if the restrictions $d|_U$ and $r|_U$ are injective.
	It follows that $d|_U$ and $r|_U$ are homeomorphisms onto their images if $U$ is a bisection since $d$ and $r$ are open maps.
	We let $\Bis(G)$ denote the set of open bisections of $G$.
	Note that $\Bis(G)$ is a basis of $G$ since $G$ is \'etale.
	An \'etale groupoid $G$ is said to be ample if $G$ has a basis of compact open bisections.
	
	An \'etale groupoid $G$ is said to be effective if $G^{(0)}$ coincides with the interior of $\Iso(G)$,
	where 
	\[
	\Iso(G)\defeq\{\alpha\in G\colon d(\alpha)=r(\alpha)\}
	\]
	is the isotropy of $G$.
	An \'etale groupoid $G$ is said to be topologically principal if 
	\[
	\{x\in G^{(0)}\mid G_x\cap G^x=\{x\}\}
	\]
	is dense in $G^{(0)}$.
	If $G$ is topologically principal,
	then $G$ is effective.
	If $G$ is second countable and effective,
	then $G$ is topologically principal (see \cite[Proposition 3.6]{renault}).
	
	An \'etale groupoid $G$ is said to be topologically transitive if $r(d^{-1}(U))$ is dense in $G^{(0)}$ for all non-empty open set $U\subset G^{(0)}$.
	Equivalently,
	$G$ is topologically transitive if and only if each non-empty open invariant subset $U\subset G^{(0)}$ is dense in $G^{(0)}$.
	If there exists $x\in G^{(0)}$ such that $r(d^{-1}(\{x\}))\subset G^{(0)}$ is dense,
	then $G$ is topologically transitive.
	The converse is true if $G$ is second countable by \cite[Lemma 3.4]{STEINBERG20192474}.
	An \'etale groupoid $G$ is said to be minimal if $G^{(0)}$ has no nontrivial
	closed invariant subsets.
	
	Following \cite[Definition 2.1]{AnatharamanDelapurelyinfinite},
	we define the notion of locally contracting groupoids.
	An \'etale groupoid $G$ is said to be locally contracting if,
	for every nonempty open subset $U\subset G^{(0)}$,
	there exist an open subset $V\subset U$ and an open bisection $S\subset G$ such that $\overline{V}\subset d(S)$ and $r(S\overline{V})\subsetneq V$.

	\subsection{Cocycles on \'etale groupoids}
	
	We recall some notions concerning cocycles on \'etale groupoids.
	See \cite[Section I.1]{renault1980groupoid} or \cite[Subsection 2.3]{Armstrong_Brownlowe_Sims_simplicity_of_twistedDeaconu} for further details.
	
	\subsubsection{A topological group of 1-cocycles}
	
	Let $G$ be a topological groupoid and $H$ be a topological abelian group.
	A groupoid homomorphism $c\colon G\to H$ is called a $H$-valued $1$-cocycle.
	We let $Z(G,H)$ denote the set of all continuous 1-cocycles $c\colon G\to H$.
	When $H=\T$,
	we simply write $Z(G)\defeq Z(G,\T)$.
	Then $Z(G,H)$ is an abelian group with respect to the pointwise product.
	In addition,
	$Z(G, H)$ is a topological group with respect to the compact-open topology,
	which is the topology generated by the subsets
	\[
	[K,U]\defeq \{c\in Z(G,H)\mid c(K)\subset U\}
	\]
	for all compact sets $K\subset G$ and open sets $U\subset H$ (see \cite[Appendix A]{Hatcher_alg_top}).
	For completeness, we show that $Z(G,H)$ is a topological group.
	We first recall some general results on the compact-open topology.
	Let $C(X, Y)$ denote the space of continuous maps from $X$ to $Y$ equipped with the compact-open topology.
	\begin{prop}\label{prop: lemma for continuity of the product of C(X,H)}
		Let $X, Y, Z,W$ be Hausdorff spaces.
		Assume that $X$ and $Y$ are locally compact.
		Then
		\[
		C(X,Z)\times C(Y,W)\ni (f,g)\mapsto f\times g\in C(X\times Y, Z\times W)
		\]
		is continuous,
		where $f\times g(x,y)\defeq (f(x), g(y))$ for $(x,y)\in X\times Y$.
	\end{prop}
	
	\begin{proof}
		Consider a compact set $M\subset X\times Y$,
		an open set $O\subset Z\times W$ and $(f,g)\in C(X,Z)\times C(Y,W)$ with $f\times g\in [M,O]$.
		For $p=(x,y)\in M$,
		there exist open sets $U_p\subset Z$ and $V_p\subset W$ such that $(f(x),g(y))\in U_p\times V_p\subset O$.
		Using the assumption that $X$ and $Y$ are locally compact,
		take compact neighborhoods $K_p$ and $L_p$ of $x$ and $y$ respectively
		so that $x\in K_p\subset f^{-1}(U_p)$ and $y\in L_p\subset g^{-1}(V_p)$.
		Then there exists a finite cover $\{K_p\times L_p\}_{p\in F}$ of $M$,
		where $F\subset M$ is a finite set.
		Put
		\[
		N\defeq \bigcap_{p\in F}[K_p, U_p]\times [L_p, V_p].
		\]
		Then $N$ is an open set of $C(X,Z)\times C(Y,W)$ with $(f,g)\in N$.
		For any $(f',g')\in N$,
		one can see $f'\times g'(M)\subset O$.
		Indeed,
		for any $q\in M$,
		there exists $p\in F$ such that $q\in K_p\times L_p$.
		Since $f'\times g'(K_p\times L_p)\subset U_p\times V_p\subset O$,
		we obtain $f'\times g'(q)\in O$.
		Thus we obtain $f'\times g'\in [M,O]$ and this proves the continuity of the map in the statement.
		\qed
	\end{proof}
	
	\begin{prop}
		Let $X$ be a locally compact Hausdorff space and $H$ be a topological group.
		Then $C(X,H)$ is a topological group with respect to the pointwise product.
	\end{prop}
	
	\begin{proof}
		It is straightforward to see that $C(X,H)$ is a group.
		The inverse of $C(X, H)$ is continuous since we have
		\[
		[K,U]^{-1}=[K, U^{-1}]
		\]
		for a compact set $K\subset X$ and an open set $U\subset H$.
		We show the continuity of the product of $C(X,H)$.
		Let $m\colon H\times H\to H$ denote the product of $H$.
		Then the map
		\[
		m_*\colon C(X\times X, H\times H)\ni h\mapsto m\circ h\in C(X\times X, H)
		\]
		is continuous,
		since we have
		\[
		m_*^{-1}([K,U])=[K, m^{-1}(U)]
		\]
		for all compact set $K\subset X\times X$ and an open set $U\subset H$.
		Define a continuous map $\iota\colon X\to X\times X$ by $\iota(x)=(x,x)$.
		Then the map
		\[
		\iota^*\colon C(X\times X, H)\ni h\mapsto h\circ \iota\in C(X,H)
		\]
		is also continuous,
		since we have
		\[
		(\iota^*)^{-1}([K,U])=[\iota(K),U]
		\]
		for all compact set $K\subset X$ and an open set $U\subset H$.
		Since the product of $C(X,H)$ is the composition map of
		\begin{align*}
			C(X, H)\times C(X,H)\ni (f,g)&\mapsto f\times g\in C(X\times X, H\times H), \\
			C(X\times X, H\times H)\ni h&\mapsto m_*(h)\in C(X\times X, H),\\
			C(X\times X, H)\ni h&\mapsto \iota^*(h)\in C(X,H)
		\end{align*}
		and these maps are continuous by Proposition \ref{prop: lemma for continuity of the product of C(X,H)} and the above argument,
		the product of $C(X,H)$ is continuous.
		Therefore $C(X,H)$ is a topological group. 
		\qed
	\end{proof}
	
	\begin{cor}
	Let $G$ be a locally compact Hausdorff topological groupoid and $H$ be a topological abelian group.
	Then $Z(G,H)$ is a closed subgroup of $C(G,H)$.
	In particular,
	$Z(G,H)$ is a topological abelian group with respect to the compact-open topology.
	\end{cor}
	\begin{proof}
		One can check that $Z(G,H)$ is a subgroup of $C(G,H)$.
		Since the evaluation map on $C(G,H)$ is continuous by \cite[Proposition A.14(a)]{Hatcher_alg_top},
		$Z(G,H)$ is closed in $C(G,H)$.
		This completes the proof.
		\qed
	\end{proof}
	
	\subsubsection{2-cocycles}
	
	Let $G$ be a locally compact Hausdorff \'etale groupoid and $H$ be a topological abelian group.
	A $H$-valued 2-cocycle on $G$ is a map $\sigma\colon G^{(2)}\to H$ such that
	\[
	\sigma(\alpha, \beta)\sigma(\alpha\beta,\gamma)=\sigma(\beta,\gamma)\sigma(\alpha,\beta\gamma)
	\]
	for all $(\alpha,\beta),(\beta,\gamma)\in G^{(2)}$.
	A 2-cocycle $\sigma\colon G^{(2)}\to H$ is said to be normalized if
	\[
	\sigma(\alpha,d(\alpha))=\sigma(r(\alpha),\alpha)=1
	\]
	for all $\alpha\in G$.
	We denote the set of all continuous normalized $H$-valued 2-cocycles on $G$ by $Z^2(G,H)$.
	Then $Z^2(G,H)$ is an abelian group with respect to the pointwise product.
	We simply write $Z^2(G)\defeq Z^2(G,\T)$ when $H=\T$.
	We note that,
	if $\sigma\in Z(G,H)$,
	then $\sigma(\alpha,\alpha^{-1})=\sigma(\alpha^{-1},\alpha)$ for $\alpha\in G$.
	
	Let $b\colon G\to H$ be a continuous map.
	We define $\partial b\colon G^{(2)}\to H$ by
	\[
	\partial b(\alpha,\beta)\defeq b(\alpha)b(\beta)b(\alpha\beta)^{-1}
	\]
	for $(\alpha,\beta)\in G^{(2)}$.
	Then $\partial b$ is a continuous $H$-valued 2-cocycles.
	If $b$ is normalized in the sense that $b|_{G^{(0)}}=1_H$,
	then $\partial b$ is normalized.
	We say that normalized 2-cocycles $\sigma,\sigma'\in Z^2(G,H)$ are cohomologous if there exists a normalized continuous map $b\colon G\to H$ such that $\sigma\sigma'^{-1}=\partial b$.

	\subsection{Twisted groupoids}
	
	Let $G$ be an \'etale groupoid.
	A twist $\Sigma$ over $G$ is a sequence
	\[
	G^{(0)}\times\T\xlongrightarrow{\iota}\Sigma\xlongrightarrow{q}G,
	\]
	where $G^{(0)}\times \T$ is a trivial group bundle with fibers $\T$,
	$\Sigma$ is a topological groupoid and $\iota$ and $q$ are continuous groupoid homomorphism such that
	\begin{itemize}
		\item $\iota|_{G^{{(0)}}\times \{1\}}$ and $q|_{\Sigma^{(0)}}$ are homeomorphisms on the unit spaces.
		We identify $G^{(0)}\times \{1\}$ and $\Sigma^{(0)}$ with $G^{(0)}$.
		\item $q^{-1}(G^{(0)})=\iota(G^{(0)}\times \T)$.
		\item for every $\delta\in \Sigma$ and $z\in \T$,
		$\iota(r(\delta),z)\delta=\delta\iota(d(\delta),z)$.
		We simply write $z\delta\defeq\iota(r(\delta),z)\delta(=\delta\iota(d(\delta),z))$.
		\item $q\colon \Sigma\to G$ is locally trivial in the sense that for every $\alpha\in G$,
		there exists an open bisection $U\subset G$ with $\alpha\in U$ and a continuous map $\phi\colon U\to \Sigma$ such that $q\circ \phi=\id_U$ and the map $\T\times U\ni(z,\alpha)\mapsto z\phi(\alpha)$ is a homeomorphism onto $q^{-1}(U)$.
	\end{itemize}
	One can see that $q\colon \Sigma\to G$ is a quotient map in the sense that the topology of $G$ coincides with the final topology induced by $q$ (see \cite[Lemma 2.7]{Armstrong_uniqueness_theorem_for_twistedgroupoid}).
	Note that $\Sigma$ is a locally compact Hausdorff if $G$ is. 
	We often denote a twist $\Sigma$ over an \'etale groupoid $G$ by $(G,\Sigma)$.
	
	\begin{ex}
		Let $G$ be an \'etale groupoid and $\sigma\in Z^2(G)$ be a normalized 2-cocycle.
		We recall a twist $\Sigma_{\sigma}$ over $G$ associated with $\sigma$.
		See \cite[Example 11.1.5]{asims} for more details.
		Put $\Sigma_{\sigma}=G\times\T$ as a topological space.
		The unit space of $\Sigma_{\sigma}$ is $\Sigma_{\sigma}^{(0)}\defeq G^{(0)}\times \{1\}$,
		which is naturally identified with $G^{(0)}$.
		The domain and range maps are defined by
		\[
		d(\alpha,z)\defeq d(\alpha), r(\alpha,z)\defeq r(\alpha)
		\]
		for $(\alpha,z)\in \Sigma_\sigma$ respectively.
		The product of $\Sigma_{\sigma}$ is defined as
		\[
		(\alpha,z)(\beta,w)\defeq (\alpha\beta,\sigma(\alpha,\beta)zw)
		\]
		for $(\alpha,z),(\beta,w)\in \Sigma_{\sigma}$ with $d(\alpha)=r(\alpha)$.
		The inversion of $(\alpha,z)\in \Sigma_{\sigma}$ is
		\[
		(\alpha,z)^{-1}=(\alpha^{-1}, \overline{\sigma(\alpha,\alpha^{-1})}\overline{z}).
		\]
		The inclusion $G^{(0)}\times \T\subset \Sigma_{\sigma}$ is defined in the obvious way and the quotient map is defined as
		\[
		q\colon \Sigma_{\sigma}\ni (\alpha,z)\mapsto \alpha\in G.
		\]
		Then $\Sigma_{\sigma}$ is a twist over $G$.
		Note that $\Sigma_{\sigma}$ has a global continuous section
		\[
		G\ni \alpha\mapsto (\alpha,1)\in\Sigma_{\sigma}
		\]
		for $q$.
		Conversely,
		if a twist $\Sigma$ has a global continuous section $S\colon G\to \Sigma$ for the quotient map $q\colon \Sigma\to G$,
		then there exists a normalized continuous 2-cocycle $\sigma\in Z^2(G)$ such that $\Sigma\simeq \Sigma_{\sigma}$.
		Indeed,
		we can construct $\sigma\in Z^2(G)$ so that
		\[
		S(\alpha)S(\beta)=\sigma(\alpha,\beta)S(\alpha\beta)
		\]
		for all $(\alpha,\beta)\in G^{(2)}$.
		See \cite[Remark 11.1.6]{asims} for more details.
	\end{ex}
	
	Next,
	we recall the standard Haar system of a twisted groupoid $(G,\Sigma)$.
	For each $x\in G^{(0)}$ and $\alpha\in G^x$,
	fix $\delta\in q^{-1}(\{\alpha\})$.
	Then the map
	\[
	\T\ni z\mapsto z\delta\in q^{-1}(\{\alpha\})
	\]
	is a homeomorphism.
	Via this homeomorphism,
	we equip $q^{-1}(\alpha)$ with the Haar probability measure of $\T$.
	Since we have
	\[
	\Sigma^x=\coprod_{\alpha\in G^x}q^{-1}(\{\alpha\}),
	\]
	we may equip $\Sigma^x$ with the disjoint union measure $\mu^x$ with $\supp\mu^x=\Sigma^x$.
	In the similar way, we equip $\Sigma_x$ with a measure $\mu_x$ with $\supp\mu_x=\Sigma_x$.
	
	Next,
	we recall the definition of twisted groupoid C*-algebras.
	Let $(G,\Sigma)$ be a twisted groupoid over a locally compact Hausdorff \'etale groupoid $G$.
	Set
	\[
	C_c(\Sigma;G)\defeq\{f\in C_c(\Sigma)\mid \text{$f(z\delta)=zf(\delta)$ for all $z\in\T$ and $\delta\in\Sigma$}\}.
	\]
	Then $C_c(\Sigma;G)$ is a $\C$-vector space.
	The convolution of $f,g\in C_c(\Sigma;G)$ is defined as
	\[
	f*g(\delta)=\int_{\Sigma}f(\epsilon)g(\epsilon^{-1}\delta)d\mu^{r(\delta)}(\epsilon)
	\]
	for $\delta\in\Sigma$.
	This convolution is computed as
	\[
	f*g(\delta)=\sum_{\alpha\in G^{r(\delta)}}f(P(\alpha))g(P(\alpha)^{-1}\delta)
	\]
	for a section $P\colon G_{d(\delta)}\to \Sigma_{d(\delta)}$ of $q|_{\Sigma_{d(\delta)}}$.
	Note that the right hand side of this formulae is independent of the choice of $P$.
	The involution of $f\in C_c(\Sigma;G)$ is defined as $f^*(\delta)\defeq \overline{f(\delta^{-1})}$ for $\delta\in\Sigma$.
	When the twist is trivial (i.e., $\Sigma=G\times \T$),
	the *-algebra $C_c(\Sigma;G)$ coincides with the usual *-algebra $C_c(G)$ associated with $G$.
	For a function $f\colon \Sigma\to \C$,
	we let $\insupp(f)\defeq f^{-1}(\C\setminus\{0\})$ denote the open support and $\supp(f)\defeq \overline{\insupp(f)}$.
	Note that $C_c(G^{(0)})$ is identified with the subalgebra
	\[
	\{f\in C_c(\Sigma;G)\mid \supp(f)\subset \iota(G^{(0)}\times \T)\}
	\]
	of $C_c(\Sigma;G)$ via the isomorphism that carries $f\in C_c(G^{(0)})$ to $\widetilde{f}\in C_c(\Sigma;G)$ defined by $\widetilde{f}(x,z)=zf(x)$ for $(x,z)\in G^{(0)}\times \T$.
	
	For $x\in G^{(0)}$,
	let $L^2(\Sigma_x;G_x)\defeq L^2(\Sigma_x;G_x,\mu_x)$ denote the set of the square-integrable $\T$-equivariant functions on $\Sigma_x$.
	Then the left regular representation $\lambda_x\colon C_c(\Sigma;G)\to B(L^2(\Sigma_x;G_x))$ is defined by the convolution product.
	Namely,
	we have $\lambda_x(f)\xi=f*\xi$ for $f\in C_c(\Sigma;G)$ and $\xi\in L^2(\Sigma_x;G_x)$ although this is a slight abuse of notation.
	The reduced norm of $f\in C_c(\Sigma;G)$ is defined as
	\[
	\lVert f\rVert_r\defeq \sup_{x\in G^{(0)}}\lVert \lambda_x(f)\rVert
	\]
	and the twisted groupoid C*-algebra $C^*_r(\Sigma; G)$ is defined to be the completion of $C_c(\Sigma;G)$ by the reduced norm.
	Note that the inclusion $C_c(G^{(0)})\subset C_c(\Sigma;G)$ extends to the inclusion $C_0(G^{(0)})\subset C^*_r(\Sigma;G)$.
	In addition,
	there exists a faithful conditional expectation $E\colon C^*_r(\Sigma;G)\to C_0(G^{(0)})$ onto $C_0(G^{(0)})$ which extends the restriction
	\[
	C_c(\Sigma; G)\ni f\mapsto f|_{G^{(0)}\times \T}\in C_c(\Sigma; G).
	\]
	
	We introduce the so-called $j$-map.
	Let $C_0(\Sigma; G)$ denotes the set of $\T$-equivariant functions in $C_0(\Sigma)$.
	By \cite[Proposition 2.8]{BrownFullerPitts2021},
	there exists an injective norm-decreasing linear map $j\colon C^*_r(\Sigma;G)\to C_0(\Sigma; G)$ which extends $\id_{C_c(\Sigma;G)}\colon C_c(\Sigma; G)\to C_c(\Sigma; G)$.
	Via this map,
	we identify an element of $C^*_r(\Sigma;G)$ with the one of $C_0(\Sigma;G)$.
	For example,
	we may define $\insupp(a)$ and $\supp(a)$ for $a\in C^*_r(\Sigma;G)$.
	
	\subsection{Inverse semigroup actions}\label{subsection: inverse semigroup actions}
	
	We recall the basic notions about inverse semigroups.
	See \cite{lawson1998inverse} or \cite{paterson2012groupoids} for more details.
	An inverse semigroup $S$ is a semigroup such that for every $s\in S$ there exists a unique $s^*\in S$ with $s=ss^*s$ and $s^*=s^*ss^*$.
	An element $s^*$ is called a generalized inverse of $s\in S$.
	By a subsemigroup of $S$,
	we mean a subset of $S$ which is closed under the product and generalized inverse of $S$.
	We denote the set of all idempotents in $S$ by $E(S)\defeq\{e\in S\mid e^2=e\}$.
	It is known that $E(S)$ is a commutative subsemigroup of $S$.
	An inverse semigroup which consists of idempotents is called a (meet) semilattice of idempotents.
	A zero element is a unique element $0\in S$ such that $0s=s0=0$ for all $s\in S$.
	An inverse semigroup with a unit is called an inverse monoid.
	A map $\varphi\colon S\to T$ between inverse semigroups $S$ and $T$ is called a semigroup homomorphism if $\varphi(st)=\varphi(s)\varphi(t)$ for all $s,t\in S$.
	Note that a semigroup homomorphism automatically preserves generalized inverses (i.e.\ $\varphi(s^*)=\varphi(s)^*$ for all $s\in S$).
	
	\begin{ex}[{\cite[Proposition 2.2.4]{paterson2012groupoids}}]
		Let $G$ be a locally compact Hausdorff \'etale groupoid.
		The set of all open bisections in $G$ is denoted by $\mathrm{Bis}(G)$.
		For $U, V\in \mathrm{Bis}(G)$,
		their product is defined by
		\[
		UV\defeq \{\alpha\beta\in G\mid \alpha\in U, \beta\in V, d(\alpha)=r(\beta)\}.
		\]
		Then $UV\in\mathrm {Bis}(G)$ and $\mathrm{Bis}(G)$ is an inverse semigroup with respect to this product.
		Note that $U^*\in \mathrm{Bis}(G)$ is given by
		\[
		U^{-1}\defeq \{\alpha^{-1}\in G\mid \alpha\in U\}.
		\]
		In addition,
		let $\Bis^c(G)$ denote the set of all compact open bisections in $G$.
		Then $\Bis^c(G)$ is a subsemigroup of $\Bis(G)$.
	\end{ex}
	
	For a topological space $X$,
	we denote by $I_X$ the set of all homeomorphisms between open subsets in $X$.
	Then $I_X$ is an inverse semigroup with respect to the product defined by the composition of maps.
	For an inverse semigroup $S$,
	an inverse semigroup action $\alpha\colon S\curvearrowright X$ is a semigroup homomorphism $S\ni s\mapsto \alpha_s\in I_X$.
	In this paper, we always assume that every action $\alpha$ satisfies $\bigcup_{e\in E(S)}\dom(\alpha_e)=X$.
	If $S$ has a zero element,
	we assume that $\dom(\alpha_0)=\emptyset$.
	
	Next,
	we recall how to construct an \'etale groupoid from an inverse semigroup action.
	Let $X$ be a locally compact Hausdorff space.
	For an action $\alpha\colon S\curvearrowright X$,
	we associate an \'etale groupoid $S\ltimes_{\alpha}X$ as the following.
	First we put the set $S*X\defeq \{(s,x) \in S\times X \mid x\in \dom(\alpha_{s^*s})\}$.
	Then we define an equivalence relation $\sim$ on $S*X$ by declaring that $(s,x)\sim (t,y)$ if
	\[
	\text{$x=y$ and there exists $e\in E(S)$ such that $x\in \dom(\alpha_e)$ and $se=te$}.  
	\]
	Set $S\ltimes_{\alpha}X\defeq S*X/{\sim}$ and denote the equivalence class of $(s,x)\in S*X$ by $[s,x]$.
	The unit space of $S\ltimes_{\alpha}X$ is $X$, where $X$ is identified with the subset of $S\ltimes_{\alpha}X$ via the injective map
	\[
	X\ni x\mapsto [e,x] \in S\ltimes_{\alpha}X, x\in \dom(\alpha_e).
	\]
	The domain map and range map are defined by
	\[
	d([s,x])=x, r([s,x])=\alpha_s(x)
	\]
	for $[s,x]\in S\ltimes_{\alpha}X$.
	The product of $[s,\alpha_t(x)],[t,x]\in S\ltimes_{\alpha}X$ is $[st,x]$.
	The inverse is $[s,x]^{-1}=[s^*,\alpha_s(x)]$.
	Then $S\ltimes_{\alpha}X$ is a groupoid in these operations.
	For $s\in S$ and an open set $U\subset \dom(\alpha_{s^*s})$,
	define 
	\[[s, U]\defeq \{[s,x]\in S\ltimes_{\alpha}X\mid x\in U\}.\]
	These sets form an open basis of $S\ltimes_{\alpha}X$.
	In these structures,
	$S\ltimes_{\alpha}X$ is a locally compact \'etale groupoid,
	although $S\ltimes_{\alpha}X$ is not necessarily Hausdorff.
	In this paper,
	we only treat inverse semigroup actions $\alpha\colon S\curvearrowright X$ such that $S\ltimes_{\alpha}X$ become Hausdorff.
	
	Let $S$ be an inverse semigroup with $0$ and $\Gamma$ be a discrete group.
	Put $S^{\times}\defeq S\setminus\{0\}$.
	A map $\theta\colon S^{\times}\to \Gamma$ is called a partial homomorphism if $\theta(st)=\theta(s)\theta(t)$ for all $s,t\in S^{\times}$ with $st\not=0$.
	Assume that $\theta\colon S^{\times}\to \Gamma$ is a partial homomorphism and $\alpha\colon S\curvearrowright X$ is an action on a topological space $X$.
	Then we associate a continuous cocycle $\widetilde{\theta}\colon S\ltimes_{\alpha} X \to \Gamma$ defined by
	\[
	\widetilde{\theta}([s,x])\defeq \theta(s)
	\]
	for all $[s,x]\in S\ltimes_{\alpha}X$.
	
	\subsection{C*-algebras}
	
	We recall basic notions about C*-algebras.
	\begin{defi}
		Let $B\subset A$ be an inclusion of C*-algebras.
		An element $n\in A$ is called a normalizer of $B$ if $nBn^*\cup n^*Bn\subset B$ holds.
	\end{defi}
	
	\begin{defi}[{\cite[Definition 5.1]{renault}}]
	Let $B\subset A$ be an inclusion of C*-algebras.
	Then $B$ is called a Cartan subalgebra if
	\begin{enumerate}
		\item $B$ contains an approximate unit of $A$,
		\item $B$ is maximal abelian in $A$,
		\item $B$ is regular, in the sense that the linear span of the normalizers of $B$ is dense in $A$,
		\item there exists a faithful conditional expectation $E\colon A\to B$.
	\end{enumerate}
	\end{defi}
	
	By \cite[Theorem 5.9]{renault} and \cite{Ali_Raad_generalization_of_Renault},
	it is known that,
	if $B\subset A$ is a Cartan subalgebra,
	then there exists a locally compact Hausdorff effective \'etale groupoid $G$ and a twist $\Sigma$ over $G$ such that the inclusion $B\subset A$ is isomorphic to $C_0(G^{(0)})\subset C^*_r(\Sigma;G)$.
	
	We recall the definition of purely infinite C*-algebras.
	Throughout this paper,
	we restrict our attention to simple C*-algebras when dealing with purely infinite C*-algebras.
	Various equivalent characterizations of purely infinite C*-algebras are known; see, for example,
	\cite[Proposition 4.1.1]{Rordamclassification}.
	
	\begin{defi}
		Let $A$ be a C*-algebra.
		A projection $p\in A$ is said to be infinite if there exists $s\in A$ such that $s^*s=p$ and $ss^*\lneq p$.
		A simple C*-algebra $A$ is said to be purely infinite if every nonzero hereditary C*-subalgebra $B\subset A$ contains an infinite projection.
	\end{defi}
	
	A simple,
	purely infinite,
	separable nuclear C*-algebra is called a Kirchberg algebra.
	The Kirchberg--Phillips classification theorem (\cite{Kirchberg1994Classification}, \cite{Phillips_classification}) asserts that unital Kirchberg algebras satisfying the UCT are classified up to *-isomorphism by their K-theory.
	
	We conclude this subsection by recalling the notion of a prime C*-algebra.
	
	\begin{defi}
		Let $A$ be a C*-algebra.
		An ideal $I\subset A$ is said to be essential if $aI=\{0\}$ implies $a=0$ for every $a\in A$.
		A C*-algebra $A$ is said to be prime if every nonzero ideal of $A$ is essential.
	\end{defi}
	
	\subsection{Miscellaneous facts}
	
	We collect facts which will be used in the later sections.
	We begin with fundamental facts about normalizers of groupoid C*-algebras.
	
	\begin{prop}[{\cite[Proposition 4.8]{renault}, \cite[Theorem 3.1]{Armstrong_Brown_local_bisection_hypothesis}}] \label{prop: characterization of normalizer}
		Let $(G,\Sigma)$ be a twisted groupoid over a locally compact Hausdorff \'etale groupoid $G$,
		$q\colon \Sigma\to G$ denote the quotient map and $n\in C^*_r(\Sigma; G)$.
		If $q(\insupp(n))$ is an open bisection of $G$,
		then $n$ is a normalizer of $C_0(G^{(0)})$.
		Conversely,
		if $n$ is a normalizer of $C_0(G^{(0)})$ and $G$ is effective,
		then $q(\insupp(n))$ is an open bisection of $G$.
	
	\end{prop}
	Note that the reduced norm of a normalizer $n\in C^*_r(\Sigma;G)$ coincides with the supremum norm,
	since we have $n^*n\in C_0(G^{(0)})$ and
	\begin{align*}
	\lVert n\rVert^2&=\lVert n^*n\rVert=\lVert n^*n\rVert_{\infty}=\sup_{x\in G^{(0)}}\lvert n^*n(x)\rvert\\
	&=\sup_{\delta\in \Sigma}\lvert n(\delta)\rvert^2=\lVert n\rVert_{\infty}^2,
	\end{align*}
	where $\lVert\cdot\rVert_{\infty}$ denotes the supremum norm.
	
	The next proposition is often used to analyze automorphisms of \'etale groupoids.
	
	\begin{prop}[{\cite[Proposition 1.5.3]{Komura2025Weyl}}]\label{prop: restriction of autG to G0 is injective}
		Let $G$ be a locally compact Hausdorff \'etale groupoid.
		Assume that $G$ is effective.
		If $\Phi\in \Aut(G)$ satisfies $\Phi|_{G^{(0)}}=\id_{G^{(0)}}$,
		then $\Phi=\id_G$.
	\end{prop}

	We show a so-called uniqueness theorem (Proposition \ref{prop: uniqueness theorem for groupoid C*-algebra}) without assuming second countability of \'etale groupoids.
	This result is not new and has already been proved in \cite[Theorem 7.29]{KwasniewskiMeyerEssentialcrossedproduct} and \cite[Proposition 4.15]{Brix_Carsen_Sims2024}.
	We include the proof here since it seems that our argument requires fewer techniques than the existing ones.
	We begin with some preliminary results.
	
	\begin{prop}\label{prop: equivalent condition to effectiveness}
		Let $G$ be a locally compact Hausdorff \'etale groupoid.
		Then $G$ is effective if and only if the following condition holds:
		for a compact set $K\subset G\setminus G^{(0)}$ and a nonempty open set $U\subset G^{(0)}$,
		there exists a nonempty open set $V\subset U$ such that $VKV=\emptyset$.
		
	\end{prop}
	
	\begin{proof}
		We first assume the latter condition and show that $G$ is effective.
		Take $\alpha\in\Iso(G)^{\circ}$ and suppose that $\alpha\not\in G^{(0)}$.
		Then there exists an open bisection $W\subset \Iso(G)$ with $\alpha\in W$ such that $\overline{W}\subset G\setminus G^{(0)}$ and $\overline{W}$ is compact.
		By the assumption,
		there exists a nonempty open subset $V\subset r(W)$ such that $V\overline{W}V=\emptyset$.
		Using $\emptyset\not= V\subset r(W)$,
		take $x\in V$ and $\gamma\in W$ with $x=r(\gamma)$.
		Since $\gamma\in W\subset \Iso(G)$,
		we have $r(\gamma)=d(\gamma)$.
		Hence $\gamma\in V\overline{W}V$ and this is a contradiction.
		
		Next,
		we assume that $G$ is effective.
		Take finitely many open bisections $\{S_i\}_{i\in I},\{S_i'\}_{i\in I}$ of $G$ such that
		\begin{itemize}
			\item $S_i\subset \overline{S_i}\subset S_i'$ and $\overline{S_i}$ is compact for each $i\in I$, and
			\item $K\subset \bigcup_{i\in I}S_i$.
		\end{itemize}
		Put
		\[
		A_{S'_i}\defeq r(S'_i\setminus \Iso(G))\cup (G^{(0)}\setminus \overline{r(S'_i)}).
		\]
		Since $r(S'_i\setminus \Iso(G))$ and $G^{(0)}\setminus \overline{r(S'_i)}$ are dense in $r(S'_i)$ and $G^{(0)}\setminus r(S'_i)$ respectively,
		$A_{S'_i}$ is open and dense in $G^{(0)}$.
		By Baire theorem,
		$\bigcap_{i\in I} A_{S'_i}$ is dense in $G^{(0)}$.
		Note that we have
		\[
		A_{S'_i}\subset \{x\in G^{(0)}\mid G_x^x\cap S'_i=\emptyset\}
		\]
		for each $i\in I$.
		
		Fix $x_0\in U\cap  \bigcap_{i\in I} A_{S'_i}$.
		Put
		\[
		I'\defeq \{i\in I\mid x_0\in r(S_i')\}.
		\]
		For each $i\in I'$,
		there exists $\gamma_i\in S_i'$ with $x_0=r(\gamma_i)$.
		Since $x_0\in A_{S'_i}$,
		we have $r(\gamma_i)\not=d(\gamma_i)$.
		Then there exists an open bisection $V_i\subset S_i'$ with $\gamma_i\in V_i$ and $r(V_i)\cap d(V_i)=\emptyset$.
		Put 
		\[
		V\defeq \bigg( U\cap \bigcap_{i\in I'} r(V_i)\bigg)\setminus\bigg( \bigcup_{i\in I\setminus I'}r(\overline{S_i})\bigg).
		\]
		Then $V\subset U$ is open and nonempty since $x_0\in V$.
		
		Now, it suffices to show $VKV=\emptyset$.
		Assume that there exists $\gamma\in VKV$.
		Then there exists $i\in I$ with $\gamma\in S_i$.
		Since $r(\gamma)\in V\cap r(S_i)$,
		$i$ belongs to $I'$.
		Since $r(\gamma)\in r(V_i)$,
		there exists $\beta\in V_i$ with $r(\gamma)=r(\beta)$.
		We may conclude $\gamma=\beta$ since $\gamma$ and $\beta$ belong to the same bisection $S_i'$.
		Since we have $\gamma=\beta\in V_i$, $d(V_i)\cap r(V_i)=\emptyset$ and $V\subset r(V_i)$,
		we obtain $d(\gamma)\not\in V$.
		This contradicts to $\gamma\in VKV$.
		Hence $VKV=\emptyset$.
		\qed
	\end{proof}
	
	\begin{lem}\label{lem: lemma for uniqueness theorem}
		Let $G$ be a locally compact Hausdorff effective \'etale groupoid and $E\colon C^*_r(G)\to C_0(G^{(0)})$ denote the standard conditional expectation.
		Fix $f\in C_c(G)$ and $\varepsilon>0$ and assume that $\pi\colon C_c(G)\to B(H)$ is a *-representation such that $\pi|_{C_c(G^{(0)})}$ is injective.
		Then there exists $h\in C_c(G^{(0)})_+$ with the following property:
		\begin{enumerate}
			\item $hfh=hE(f)h$,
			\item $\lVert \pi(hfh)\rVert\geq \lVert E(f)\rVert-\varepsilon$, and
			\item $\lVert h\rVert=1$.
		\end{enumerate}
		In particular,
		$\lVert\pi(f)\rVert\geq \lVert E(f)\rVert$ holds for all $f\in C_c(G)$.
	\end{lem}
	\begin{proof}
		Put
		\begin{align*}
			&K\defeq \supp(f-E(f)),\\
			&U\defeq \{x\in G^{(0)}\mid \lvert f(x)\rvert>\lVert E(f)\rVert-\varepsilon\}.
		\end{align*}
		Since $K\subset G\setminus G^{(0)}$ and $U\subset G^{(0)}$ is a nonempty open set,
		there exists a nonempty open set $V\subset U$ such that $VKV=\emptyset$ by Proposition \ref{prop: equivalent condition to effectiveness}.
		Take $h\in C_c(V)_+$ with $\lVert h\rVert=1$ by Urysohn's lemma.
		Then we obtain $hfh=hE(f)h$ from $h(f-E(f))f=0$.
		In addition,
		since $h(x)=1$ for some $x\in V$,
		we have
		\[
		\lVert hE(f)h\rVert\geq \lvert (hE(f)h)(x)\rvert=\lvert f(x)\rvert>\lVert E(f)\rVert-\varepsilon.
		\]
		Since $\pi$ is injective on $C_c(G^{(0)})$,
		we obtain
		\[
		\lVert \pi(hfh)\rVert=\lVert\pi(hE(f)h)\rVert=\lVert hE(f)h\rVert>\lVert E(f)\rVert-\varepsilon.
		\]
		The last assertion is now obvious.
		This completes the proof.
		\qed
	\end{proof}
	
	Now,
	we are ready to show the uniqueness theorem for $C_0(G^{(0)})\subset C^*_r(G)$.
	
	\begin{prop}\label{prop: uniqueness theorem for groupoid C*-algebra}
		Let $G$ be an \'etale locally compact Hausdorff groupoid.
		Assume that $G$ is effective and $\pi\colon C^*_r(G)\to B(H)$ is a *-representation such that $\pi|_{C_0(G^{(0)})}$ is injective.
		Then $\pi$ is injective.
	\end{prop}
	
	\begin{proof}
		Let $E\colon C^*_r(G)\to C_0(G^{(0)})$ denote the standard conditional expectation.
		By Lemma \ref{lem: lemma for uniqueness theorem},
		we have a conditional expectation $F\colon \pi(C^*_r(G))\to \pi(C_0(G^{(0)}))$ such that $F\circ \pi=\pi\circ E$.
		Assume that $a\in C^*_r(G)$ satisfies $\pi(a)=0$.
		Then we have $F(\pi(a^*a))=0$ and hence $\pi(E(a^*a))=0$.
		Since $\pi$ is injective on $C_0(G^{(0)})$,
		we have $E(a^*a)=0$.
		Since $E$ is faithful,
		we obtain $a=0$ and therefore $\pi$ is injective.
		\qed
	\end{proof}
	
	Proposition \ref{prop: uniqueness theorem for groupoid C*-algebra} can be rephrased as the following corollary,
	which is known as the ideal intersection property of $C_0(G^{(0)})\subset C^*_r(G)$.
	
	\begin{cor}\label{cor: ideal intersection property}
		Let $G$ be an \'etale locally compact Hausdorff groupoid.
		Assume that $G$ is effective and $I\subset C^*_r(G)$ is a nonzero ideal.
		Then $I\cap C_0(G^{(0)})$ is also a nonzero ideal of $C_0(G^{(0)})$.
	\end{cor}
	
	The ideal intersection property has various applications in the structural analysis of groupoid C*-algebras.
	For example,
	\cite{Brown2014} studies the simplicity of groupoid C*-algebras using the ideal intersection property.
	As another example,
	we provide a sufficient condition for a groupoid C*-algebra to be prime.
	The property of primeness is important and plays a crucial role in this paper,
	for instance in Theorem \ref{theorem every compact subgroup of Cartan fixing are fin gen}.
	
	\begin{cor}[{cf.\ \cite[Proposition 1.5.1]{Komura2025Weyl}}]\label{cor: prime is equivalent to topologically transitive}
		Let $G$ be a locally compact Hausdorff \'etale groupoid.
		If $C^*_r(G)$ is prime,
		then $G$ is topologically transitive.
		If $G$ is topologically transitive and effective,
		then $C^*_r(G)$ is prime.
	\end{cor}
	
	\begin{proof}
		The first assertion follows from \cite[Proposition 1.5.1]{Komura2025Weyl}.
		Assume that $G$ is topologically transitive and effective.
		Consider a nonzero ideal $I\subset C^*_r(G)$ and assume that $a\in C^*_r(G)$ satisfies $aI=\{0\}$.
		There exists an open set $U\subset G^{(0)}$ with $I\cap C_0(G^{{0}})=C_0(U)$.
		Then $U$ is invariant by \cite[Lemma 10.3.1]{asims} and nonempty by Corollary \ref{cor: ideal intersection property}.
		Thus $U\subset G^{(0)}$ is dense since we assume that $G$ is topologically transitive.
		Now one can check $a=0$ in the same way as \cite[Proposition 1.5.1]{Komura2025Weyl}.
		Indeed,
		since we have $a^*aC_0(U)=\{0\}$,
		we obtain $a^*a(x)=0$ for all $x\in U$.
		Since $U\subset G^{(0)}$ is dense,
		we obtain $E(a^*a)=0$,
		where $E\colon C^*_r(G)\to C_0(G^{(0)})$ is the standard conditional expectation.
		Since $E$ is faithful,
		we obtain $a=0$.
		Therefore $I$ is an essential ideal and $C^*_r(G)$ is prime.
		\qed
	\end{proof}
	
	\begin{rem}
		In \cite[Proposition 1.5.1]{Komura2025Weyl},
		the author assumes that $G$ is topologically transitive and topologically principal to show that $C^*_r(G)$ is prime.
		In Corollary \ref{cor: prime is equivalent to topologically transitive},
		we relax the assumption that $G$ is topologically principal to effectiveness of $G$.
	\end{rem}
		
		Following \cite{AnatharamanDelapurelyinfinite},
		we provide a sufficient condition for groupoid C*-algebras to be purely infinite
		
	\begin{thm}[{\cite[Proposition 2.4]{AnatharamanDelapurelyinfinite}, \cite[Theorem 10.4.2]{asims}}] \label{thm: locally contracting implies purely infinite}
		Let $G$ be a locally compact Hausdorff \'etale groupoid.
		Assume that $G$ is effective and locally contracting.
		Then $C^*_r(G)$ is purely infinite.
	\end{thm}
	
	\begin{rem}
		In \cite[Proposition 2.4]{AnatharamanDelapurelyinfinite},
		the author assumes that $G$ is topologically principal (referred to as essentially free in \cite{AnatharamanDelapurelyinfinite}),
		rather than effective.
		The assumption that $G$ is topologically principal is used to apply \cite[Lemma 2.3]{AnatharamanDelapurelyinfinite}.
		However,
		this result remains valid under the weaker assumption that $G$ is effective by Proposition \ref{prop: equivalent condition to effectiveness}.
		Thus,
		the assumption of topological principality can be relaxed to effectiveness.
	\end{rem}

	\section{Automorphism groups of twisted groupoids}
	
	\subsection{Automorphism and cohomology class of twisted groupoid}
	
	In this section,
	when we refer to $(G,\Sigma)$ as a twisted groupoid,
	we assume that $G$ is a locally compact Hausdorff \'etale groupoid and $\Sigma$ is a twist over $G$.
	The main purpose of this section is to investigate the automorphism group $\Aut(\Sigma;G)$ of a twisted groupoid $(G,\Sigma)$ (Theorem \ref{theorem: main theorem:short exact sequence}).
	Combined with Taylor's result (Theorem \ref{thm: J.Taylors theorem}),
	we obtain a structure theorem for Cartan-preserving automorphism groups (Corollary \ref{cor: C*-analogue of Feldman moore}).

	\begin{defi}
		Let $(G,\Sigma)$ be a twisted groupoid.
		An automorphism of $(G,\Sigma)$ is an automorphism $\Phi\colon \Sigma\to \Sigma$ as a topological groupoid such that
		\[
		\Phi(z\delta)=z\Phi(\delta)
		\]
		for all $z\in\T$ and $\delta\in\Sigma$.
		We denote the set of all automorphisms of $(G,\Sigma)$ by $\Aut(\Sigma;G)$.
	\end{defi}
	
	\begin{rem}\label{rem def of widetilde{Phi}}
		For $\Phi\in\Aut(\Sigma;G)$,
		we have
		\[
		\Phi|_{G^{(0)}\times \T}=\Phi|_{G^{(0)}}\times \id_{\T}.
		\]
		In particular,
		since $\Phi$ maps $G^{(0)}\times \T$ onto itself,
		there exists a unique $\widetilde{\Phi}\in \Aut(G)$ such that the following diagram commutes:
		
		\begin{center}
			\begin{tikzpicture}[auto,scale=2]
				\node (Sig) at (0,0) {$\Sigma$};
				\node (G) at (2,0){$G$};
				\node (G') at (2,-1){$G$};
				\node (Sig') at (0,-1) {$\Sigma$};
				\node (torus) at (-2, 0) {$G^{(0)}\times\T$};
				\node (torus') at (-2, -1) {$G^{(0)}\times\T$};
				\draw[->] (Sig) to node {$\Phi$} (Sig') ;
				\draw[->] (Sig) to node {$q$} (G);
				\draw[->,swap] (Sig') to node {$q$} (G');
				\draw[->,swap] (torus) to node {$\Phi|_{G^{(0)}}\times \id_{\T}$} (torus');
				\draw[->] (torus) to node {$\iota$} (Sig);
				\draw[->,swap] (torus') to node {$\iota$} (Sig');
				\draw[->] (G) to node {$\widetilde{\Phi}$}(G');
			\end{tikzpicture}
			,
		\end{center}
		 where $q$ and $\iota$ denote the quotient and inclusion maps respectively.
		 One can see that the map
		 \[
		 \Aut(\Sigma;G)\ni \Phi\mapsto \widetilde{\Phi}\in \Aut(G)
		 \]
		 is a group homomorphism.
	\end{rem}
	
	\begin{defi}\label{defi: definition of cohomologous}
		Let $G$ be a locally compact Hausdorff \'etale groupoid.
		Assume that $\Sigma$ and $\Sigma'$ are twists over $G$.
		Then $\Sigma$ and $\Sigma'$ are said to be cohomologous if there exists an isomorphism $\Phi\colon\Sigma\to \Sigma'$ such that the following diagram commutes:
		
		\begin{center}
			\begin{tikzpicture}[auto,scale=1.5]
				\node (Sig) at (0,0) {$\Sigma$};
				\node (G) at (2,-1){$G$};
				\node (Sig') at (0,-2) {$\Sigma'$};
				\node (torus) at (-2, -1) {$G^{(0)}\times\T$};
				\draw[->] (Sig) to node {$\Phi$} (Sig') ;
				\draw[->] (Sig) to node {$q$} (G);
				\draw[->,swap] (Sig') to node {$q'$} (G);
				\draw[->] (torus) to node {$\iota$} (Sig);
				\draw[->,swap] (torus) to node {$\iota'$} (Sig');
			\end{tikzpicture}
			,
		\end{center}
		where $\iota$, $\iota'$ are the inclusion maps and $q$, $q'$ are the quotient maps.
	\end{defi}
	
	\begin{rem}
		We remark that,
		for $\Phi\in \Aut(\Sigma;G)$,
		the twist $\Sigma$ is cohomologous to itself via $\Phi$ if and only if $\widetilde{\Phi}$ in Remark \ref{rem def of widetilde{Phi}} is $\id_G$.
		In this case,
		we will later show that $\Phi\colon \Sigma\to \Sigma$ is of the form
		\[
		\Phi(\delta)=c(q(\delta))\delta,\,\,\,(\delta\in\Sigma)
		\]
		for some $c\in Z(G)$ in Theorem \ref{theorem: main theorem:short exact sequence}.
		
		We make another remark.
		Assume that $\Phi\colon \Sigma\to \Sigma'$ is an isomorphism as topological groupoids.
		If $G$ is effective and the left triangle of the diagram in Definition \ref{defi: definition of cohomologous} commutes,
		then $\Sigma$ and $\Sigma'$ are cohomologous via $\Phi$.
		Indeed,
		since $\Phi$ maps $G^{(0)}\times \T$ onto itself,
		we have $\widetilde{\Phi}\in\Aut(G)$ with
		\begin{center}
			\begin{tikzpicture}[auto,scale=2]
				\node (Sig) at (0,0) {$\Sigma$};
				\node (G) at (2,0){$G$};
				\node (G') at (2,-1){$G$};
				\node (Sig') at (0,-1) {$\Sigma'$};
				\node (torus) at (-2, -0.5) {$G^{(0)}\times\T$};
				\draw[->] (Sig) to node {$\Phi$} (Sig') ;
				\draw[->] (Sig) to node {$q$} (G);
				\draw[->,swap] (Sig') to node {$q'$} (G');
				\draw[->] (torus) to node {$\iota$} (Sig);
				\draw[->,swap] (torus) to node {$\iota'$} (Sig');
				\draw[->] (G) to node {$\widetilde{\Phi}$}(G');
			\end{tikzpicture}
			.
		\end{center}
		Since $\widetilde{\Phi}|_{G^{(0)}}=\id_{G^{(0)}}$,
		we obtain $\widetilde{\Phi}=\id_G$ by Proposition \ref{prop: restriction of autG to G0 is injective} if $G$ is effective.
		Hence the right triangle of the diagram in Definition \ref{defi: definition of cohomologous}  commutes.
		Thus, $\Sigma$ and $\Sigma'$ are cohomologous via $\Phi$.
	\end{rem}
	
	The notion of cohomology for twisted groupoids is compatible with that for 2-cocycles.
	
	\begin{prop}\label{prop: equivalence of cohomologous of twists and 2-cocycles}
		Let $G$ be a locally compact Hausdorff \'etale groupoid and $\Sigma_{\sigma}\defeq G\times_{\sigma}\T$ denote the twist over $G$ arising from a normalized 2-cocycle $\sigma\in Z^2(G)$.
		Then,
		for normalized 2-cocycles $\sigma, \sigma'\in Z^2(G)$,
		$\Sigma_{\sigma}$ and $\Sigma_{\sigma'}$ are cohomologous as twists if and only if $\sigma$ and $\sigma'$ are cohomologous as 2-cocycles.
	\end{prop}
	
	\begin{proof}
		Assume that $\sigma$ and $\sigma'$ are cohomologous.
		Then there exists a normalized continuous function $b\colon G\to \T$ such that $\sigma\sigma'^{-1}=\partial b$.
		Define a map $\Phi\colon \Sigma_{\sigma}\to \Sigma_{\sigma'}$ by
		\[
		\Phi(\alpha,z)\defeq (\alpha,b(\alpha)z)
		\]
		for $(\alpha, z)\in\Sigma_{\sigma}$.
		One can see that $\Phi$ is a groupoid homomorphism.
		Indeed,
		for $(\alpha,z), (\beta, w)\in \Sigma_{\sigma}$ with $d(\alpha)=r(\beta)$,
		we have
		\begin{align*}
			&\Phi((\alpha,z)(\beta, w))=\Phi(\alpha\beta, \sigma(\alpha,\beta)zw)=(\alpha\beta, b(\alpha\beta)\sigma(\alpha,\beta)zw),\\
			&\Phi(\alpha,z)\Phi(\beta,w)=(\alpha, b(\alpha)z)(\beta, b(\beta)w)=(\alpha\beta,\sigma'(\alpha,\beta)b(\alpha)b(\beta)zw).
		\end{align*}
		By $\sigma\sigma'^{-1}=\partial b$,
		we have
		\[
		\sigma(\alpha,\beta)=b(\alpha)\overline{b(\alpha\beta)}b(\beta)\sigma'(\alpha,\beta).
		\]
		Using this formula,
		we obtain 
		\[
		\Phi((\alpha,z)(\beta, w))=\Phi(\alpha,z)\Phi(\beta,w).
		\]
		In addition,
		we have
		\begin{align*}
	&	\Phi((\alpha,z)^{-1})=\Phi(\alpha^{-1}, \overline{\sigma(\alpha,\alpha)}\overline{z})=(\alpha^{-1},b(\alpha^{-1})\overline{\sigma(\alpha,\alpha)}\overline{z}),\\
	&	\Phi(\alpha,z)^{-1}=(\alpha, b(\alpha)z)^{-1}=(\alpha^{-1}, \overline{\sigma'(\alpha,\alpha^{-1})} \overline{b(\alpha)}\overline{z}).
		\end{align*}
	Using
	\[
	\sigma(\alpha, \alpha^{-1})\overline{b(\alpha^{-1})}=\sigma'(\alpha,\alpha^{-1})b(\alpha)\overline{b(\alpha\alpha^{-1})}
	\]
	and $b(\alpha\alpha^{-1})=1$,
	we obtain
	\[
	\Phi((\alpha,z)^{-1})=\Phi(\alpha,z)^{-1}.
	\]
	Now,
	one can check that $\Sigma_{\sigma}$ and $\Sigma_{\sigma'}$ are cohomologous via $\Phi$.
	
	Conversely,
	assume that $\Sigma_{\sigma}$ and $\Sigma_{\sigma'}$ are cohomologous.
	Take an isomorphism $\Phi\colon\Sigma_{\sigma}\to \Sigma_{\sigma'}$ with the following commutative diagram:
	\begin{center}
		\begin{tikzpicture}[auto,scale=1.5]
			\node (Sig) at (0,0) {$\Sigma_{\sigma}$};
			\node (G) at (2,-1){$G$};
			\node (Sig') at (0,-2) {$\Sigma_{\sigma'}$};
			\node (torus) at (-2, -1) {$G^{(0)}\times\T$};
			\draw[->] (Sig) to node {$\Phi$} (Sig') ;
			\draw[->] (Sig) to node {$q$} (G);
			\draw[->,swap] (Sig') to node {$q'$} (G);
			\draw[->] (torus) to node {$\iota$} (Sig);
			\draw[->,swap] (torus) to node {$\iota'$} (Sig');
		\end{tikzpicture}
		.
	\end{center}
	First,
	we claim that,
	for $\alpha\in G$,
	there exists a unique $b(\alpha)\in \T$ such that
	\[
	\Phi(\alpha, z)=(\alpha, b(\alpha)z)
	\]
	for all $z\in \T$.
	Since we have
	\[
	q'(\Phi(\alpha,1))=q(\alpha,1)=\alpha=q'((\alpha,1)),
	\]
	there exists a unique $b(\alpha)\in\T$ with
	\[
	\Phi(\alpha,1)=b(\alpha)(\alpha,1)=(\alpha, b(\alpha)).
	\]
	For $z\in \T$,
	we have
	\[
	\Phi(\alpha,z)=z\Phi(\alpha,1)=z(\alpha, b(\alpha))=(\alpha, b(\alpha)z).
	\]
	Now,
	we check $\sigma\sigma'^{-1}=\partial b$.
	Indeed,
	for $(\alpha,\beta)\in G^{(2)}$,
	we have
	\begin{align*}
		&\Phi((\alpha,1)(\beta,1))=(\alpha, b(\alpha))(\beta, b(\beta))=(\alpha\beta, b(\alpha)b(\beta)\sigma'(\alpha,\beta))\\
		=&\Phi(\alpha\beta,\sigma(\alpha,\beta))=(\alpha\beta,b(\alpha\beta)\sigma(\alpha,\beta)).
	\end{align*}
	Thus we obtain
	\[
	b(\alpha)b(\beta)\sigma'(\alpha,\beta)=b(\alpha\beta)\sigma(\alpha,\beta)
	\]
	and hence $\sigma\sigma'^{-1}=\partial b$.
	One can observe that $b\colon G\to \T$ is a continuous function satisfying $b|_{G^{(0)}}=1$ by the formula
	\[\Phi(\alpha, 1)=(\alpha, b(\alpha))\]
	for $\alpha\in G$.
	This completes the proof.
	\qed
	\end{proof}
	
	\begin{defi}
		Let $G$ be a locally compact Hausdorff \'etale groupoid.
		The set of the cohomology class of twists over $G$ is denoted by $\Tw(G)$.
	\end{defi}
	
	With a slight abuse of notation,
	we write $\Sigma$ for both a twist over $G$ and its cohomology class.
	
	\begin{rem}
		We note that $\Tw(G)$ is a group under the Baer sum (see \cite[Remark 2 in Section 4]{Kumjian} and the discussion following Definition 2.2 in \cite{Farsi_Gillaspy_twist_over_etale_groupoid} for more details).
		For $\Sigma,\Sigma'\in \Tw(G)$,
		the Baer sum $\Sigma\oplus\Sigma'\in\Tw(G)$ is defined as follows:
		Let $\iota,\iota',q, q'$ be the following structure maps:
		\begin{align*}
			&G^{(0)}\times \T\xrightarrow{\iota}\Sigma\xrightarrow{q}G,\\
			&G^{(0)}\times \T\xrightarrow{\iota'}\Sigma'\xrightarrow{q'}G.
		\end{align*}
		Put
		\[
			\Sigma\times_{q\times q'}\Sigma'\defeq \{(\delta, \delta')\in \Sigma\times \Sigma'\mid q(\delta)=q'(\delta')\}.
		\]
		Define an action $\T\curvearrowright \Sigma\times_{q\times q'}\Sigma'$ by
		\[
		z(\delta,\delta')=(z\delta, \overline{z}\delta')
		\]
		for $z\in\T$.
		Define
		\[
		\Sigma\oplus\Sigma'\defeq (\Sigma\times_{q\times q'}\Sigma')/\T.
		\]
		Let $Q\colon \Sigma\times_{q\times q'} \Sigma'\to \Sigma\oplus \Sigma'$ denote the quotient map and we write
		\[
		[\delta,\delta']\defeq Q(\delta,\delta')
		\]
		for $(\delta,\delta')\in \Sigma\times_{q\times q'} \Sigma'$.
		Then the structure maps of $\Sigma\oplus\Sigma'$ are
		\[
		G^{(0)}\times\T\xlongrightarrow{Q\circ (\iota,\iota')}\Sigma\oplus \Sigma' \xlongrightarrow{q}G,
		\]
		where $(\iota,\iota')$ is the following map
		\[
		(\iota,\iota')\colon G^{(0)}\times \T\ni w\mapsto (\iota(w),\iota'(w))\in \Sigma\times_{q\times q'} \Sigma'
		\]
		and $q([\delta,\delta'])\defeq q(\delta)$ for $[\delta,\delta']\in \Sigma\oplus \Sigma'$,
		although this is a slight abuse of notation.
		These operations are well-defined and it is known that $\Tw(G)$ is a group under this operation.
		We note that the inverse of $\Sigma\in\Tw(G)$ is given by
		\[
		G^{(0)}\times\T\xlongrightarrow{\iota\circ(\id_{G^{(0)}}\times \tau)}\Sigma\xlongrightarrow{q}G,
		\]
		where $\tau\in\Aut(\T)$ is the complex conjugation of $\T$,
		given by $\tau(z)=\overline{z}$ for $z\in \T$.
	\end{rem}
	
	\begin{defi}
		Let $(G, \Sigma)$ be a twisted groupoid and $\Psi\in\Aut(G)$.
		We define a twist $\Sigma_{\Psi}$ over $G$ as follows.
		As a topological groupoid,
		we set $\Sigma_{\Psi}\defeq \Sigma$.
		The structure maps of $\Sigma_{\Psi}$ are defined by
		\begin{center}
		\begin{tikzpicture}[auto,scale=2]
			\node (Sig) at (0,0) {$\Sigma$};
			\node (G) at (2,0) {$G$};
			\node (torus) at (-2,0) {$G^{(0)}\times \T$};
			\draw[->] (Sig) to node {$\Psi\circ q$} (G);
			\draw[->] (torus) to node {$\iota\circ (\Psi|_{G^{(0)}}^{-1}\times \id_{\T})$} (Sig);
		\end{tikzpicture}
		,
		\end{center}
		where $q$ and $\iota$ denote the quotient and inclusion maps of $(G,\Sigma)$ respectively.
		Namely,
		the quotient and inclusion maps of $\Sigma_{\Psi}$ are $\Psi\circ q$ and $\iota\circ (\Psi|_{G^{(0)}}^{-1}\times \id_{\T})$ respectively.
	\end{defi}

	\begin{rem}\label{rem: aut action on twist of cocycle is pull back cocycle}
		If $\Sigma\defeq \Sigma_{\sigma}$ arises from a normalized 2-cocycle $\sigma\in Z^2(G)$,
		then,
		for $\Psi\in\Aut(G)$,
		we have 
		\[(\Sigma_{\sigma})_{\Psi}\simeq \Sigma_{\sigma\circ (\Psi^{-1}\times \Psi^{-1})}.\]
		Indeed,
		take a continuous section $S\colon G\to \Sigma_{\sigma}$ of $q\colon \Sigma_{\sigma}\to G$.
		Note that we have
		\[
		S(\alpha)S(\beta)=\sigma(\alpha,\beta)S(\alpha\beta)
		\]
		for every composable pair $(\alpha,\beta)\in G^{(2)}$.
		Since the quotient map of $(\Sigma_{\sigma})_{\Psi}$ is
		\[
		\Psi\circ q\colon \Sigma_{\sigma}\to G,
		\]
		its section is given by $S'\defeq S\circ \Psi^{-1}\colon G\to (\Sigma_{\sigma})_{\Psi}$.
		For $(\alpha,\beta)\in G^{(2)}$,
		we have	
		\begin{align*}
		S'(\alpha)S'(\beta)&=S(\Psi^{-1}(\alpha))S(\Psi^{-1}(\beta)) \\
		&=\sigma(\Psi^{-1}(\alpha), \Psi^{-1}(\beta))S(\Psi^{-1}(\alpha\beta))\\
		&=\sigma(\Psi^{-1}(\alpha),\Psi^{-1}(\beta))S'(\alpha\beta).
		\end{align*}
		Hence we obtain
		\[(\Sigma_{\sigma})_{\Psi}\simeq \Sigma_{\sigma\circ(\Psi^{-1}\times \Psi^{-1})}.\]	
		Thus,
		$\Sigma_{\Psi}$ may be viewed as a generalization of the pull-back of a normalized 2-cocycle along $\Psi^{-1}\times \Psi^{-1}$.
	\end{rem}
	
	\begin{prop}\label{prop: left action of AutG on TwG}
		For $\Psi_1,\Psi_2\in \Aut(G)$ and $\Sigma\in\Tw(G)$,
		it follows
		\[
		\Sigma_{\Psi_2\circ \Psi_1}=(\Sigma_{\Psi_1})_{\Psi_2}
		\]
		in $\Tw(G)$.
		In particular,
		the map
		\[
		\Aut(G)\times \Tw(G)\ni (\Psi,\Sigma)\mapsto \Sigma_{\Psi}\in \Tw(G)
		\]
		defines a left action $\Aut(G)\curvearrowright \Tw(G)$.
	\end{prop}
	
	\begin{proof}
		Take $\Psi_1,\Psi_2\in \Aut(G)$ and $\Sigma\in\Tw(G)$ arbitrarily.
		Let $\iota, q$ denote the inclusion and quotient map associated with $\Sigma$ respectively.
		Since $\Sigma_{\Psi_2}$ is
			\begin{center}
			\begin{tikzpicture}[auto,scale=2]
				\node (Sig) at (0,0) {$\Sigma$};
				\node (G) at (2,0) {$G$};
				\node (torus) at (-2,0) {$G^{(0)}\times \T$};
				\draw[->] (Sig) to node {$\Psi_2\circ q$} (G);
				\draw[->] (torus) to node {$\iota\circ (\Psi_2|_{G^{(0)}}^{-1}\times \id_{\T})$} (Sig);
			\end{tikzpicture}
			,
		\end{center}
		one can see that $(\Sigma_{\Psi_2})_{\Psi_1}$ is given by
				\begin{center}
				\begin{tikzpicture}[auto,scale=2]
					\node (Sig) at (0,0) {$\Sigma$};
					\node (G) at (2,0) {$G$};
					\node (torus) at (-3,0) {$G^{(0)}\times \T$};
					\draw[->] (Sig) to node {$\Psi_1\circ\Psi_2\circ q$} (G);
					\draw[->] (torus) to node {{\small$\iota\circ ((\Psi_1\circ\Psi_2|_{G^{(0)}})^{-1}\times \id_{\T})$}} (Sig);
				\end{tikzpicture}
				.
			\end{center}
			This is nothing but $\Sigma_{\Psi_1\circ\Psi_2}$.
			\qed
	\end{proof}

	\begin{prop}
		For $\Psi\in \Aut(G)$,
		it follows
		\[
		(\Sigma\oplus \Sigma')_{\Psi}=\Sigma_{\Psi}\oplus\Sigma'_{\Psi}
		\]
		in $\Tw(G)$.
		In particular,
		\[\Tw(G)\ni \Sigma\mapsto \Sigma_{\Psi}\in \Tw(G)\]
		defines a group automorphism on $\Tw(G)$ and the left action in Proposition \ref{prop: left action of AutG on TwG} is an action by group automorphisms.
		 
	\end{prop}
	
	\begin{proof}
		Let $\iota$ and $q$ denote the inclusion and quotient map of $\Sigma$.
		In addition,
		let $\iota'$ denote the inclusion map of $\Sigma'$.
		Then $(\Sigma\oplus \Sigma')_{\Psi}$ is
		\[
		G^{(0)}\times \T\xlongrightarrow{Q\circ(\iota,\iota')\circ\Psi^{-1}}(\Sigma\oplus \Sigma')_{\Psi}\xlongrightarrow{\Psi\circ q} G
		\]
		and $\Sigma_{\Psi}\oplus \Sigma'_{\Psi}$ is
		\[
		G^{(0)}\times \T\xlongrightarrow{Q\circ (\iota\circ\Psi^{-1},\iota'\circ\Psi^{-1})}\Sigma_{\Psi}\oplus \Sigma'_{\Psi}\xlongrightarrow{\Psi\circ q} G.
		\]
		Comparing the above structure maps,
		we see that,
		via the following map
		\[
		\Phi\colon (\Sigma\oplus\Sigma')_{\Psi}\ni[\delta,\delta']\mapsto [\delta,\delta']\in \Sigma_{\Psi}\oplus \Sigma_{\Psi}',
		\]
		$(\Sigma\oplus\Sigma')_{\Psi}$ and $\Sigma_{\Psi}\oplus \Sigma_{\Psi}'$ are cohomologous.
		\qed		
	\end{proof}

	\begin{defi}
		For $\Sigma\in\Tw(G)$,
		the stabilizer group of the action $\Aut(G)\curvearrowright \Tw(G)$ at $\Sigma$ is denoted by $\Aut(G)_{\Sigma}$.
	\end{defi}

	\begin{rem}\label{rem: aut(G)Sigma sigma coincides with cocycle preserving automorphisms}
		Consider the twist $\Sigma_{\sigma}$ arising from a normalized 2-cocycle $\sigma\in Z^2(G)$.
		Then $\Aut(G)_{\Sigma_{\sigma}}$ coincides with
		\[
		\Aut(G)_{\sigma}\defeq \{\Psi\in \Aut(G)\mid \text{$\sigma$ and $\sigma\circ(\Psi^{-1}\times \Psi^{-1})$ are cohomologous}\}.
		\]
		Indeed,
		for $\Psi\in\Aut(G)$,
		the twists $(\Sigma_{\sigma})_{\Psi}$ and $\Sigma_{\sigma\circ(\Psi^{-1}\times \Psi^{-1})}$ are cohomologous by Remark \ref{rem: aut action on twist of cocycle is pull back cocycle}.
		Therefore, $\Sigma_{\sigma}$ and $(\Sigma_{\sigma})_{\Psi}$ are cohomologous if and only if $\sigma$ and $\sigma\circ(\Psi^{-1}\times \Psi^{-1})$ are cohomologous as 2-cocycles by Proposition \ref{prop: equivalence of cohomologous of twists and 2-cocycles}.
		Thus we obtain
		\[\Aut(G)_{\Sigma_{\sigma}}=\Aut(G)_{\sigma}.\]
		We remark that the group $\Aut(G)_{\sigma}$ is an analogue of the automorphism groups considered in \cite[Section 4]{Feldman_Moore_II},
		where the authors work in the setting of von Neumann algebras.
		Indeed,
		for a Borel equivalence relation $R$ and the cohomology class of a 2-cocycle $\sigma$ on $R$,
		the authors of \cite{Feldman_Moore_II} defined in Section 4 the group $N(R,\sigma)$ of automorphisms of $R$ that fix the cohomology class of $\sigma$.
		Thus $\Aut(G)_{\sigma}$ can be regarded as a topological analogue of $N(R,\sigma)$.
	\end{rem}
	
	\begin{prop}\label{prop: group hom from ZG to Aut(Sig G)}
		Let $(G,\Sigma)$ be a twisted groupoid.
		For $c\in Z(G)$,
		define a map $\Phi_c\colon \Sigma\to \Sigma$ by
		\[
		\Phi_c(\delta)\defeq c(q(\delta))\delta
		\]
		for $\delta\in\Sigma$,
		where $q\colon \Sigma\to G$ denotes the quotient map of $\Sigma$.
		Then $\Phi_c\in \Aut(\Sigma;G)$ and $\Sigma$ is cohomologous to itself via $\Phi_c$.
		In addition,
		the map
		\[
		i\colon Z(G)\ni c\mapsto \Phi_c\in  \Aut(\Sigma;G)
		\]
		is an injective group homomorphism.
	\end{prop}
	\begin{proof}
		This follows from a straightforward calculation.
		For example,
		one can observe
		\[
		\Phi_c(z\delta)=c(q(z\delta))z\delta=c(q(\delta))z\delta=z\Phi_c(\delta)
		\]
		for $\delta\in \Sigma$ and $z\in \T$ to show $\Phi_c\in \Aut(\Sigma;G)$.
		\qed
	\end{proof}
	
	Now,
	we are ready to show our first main theorem.
		
	\begin{thm}\label{theorem: main theorem:short exact sequence}
		Let $(G,\Sigma)$ be a twisted groupoid.
		Then the group homomorphism
		\[
		Q\colon \Aut(\Sigma;G)\ni \Phi\mapsto \widetilde{\Phi}\in\Aut(G)
		\]
		in Remark \ref{rem def of widetilde{Phi}} induces the following short exact sequence of groups:
			\begin{center}
			\begin{tikzpicture}[auto,scale=1.5]
				\node (AutSigG) at (0,0) {$\Aut(\Sigma;G)$};
				\node (AutGsig) at (2,0) {$\Aut(G)_{\Sigma}$};
				\node (ZG) at (-2,0) {$Z(G)$};
				\node (right0) at (4,0) {$0$};
				\node (left0) at (-4,0) {$0$};
				\draw[->] (AutSigG)  to node {{\small $Q$}} (AutGsig);
				\draw[->] (ZG) to node {{\small $i$}} (AutSigG);
				\draw[->] (AutGsig) to (right0);
				\draw[->] (left0) to (ZG);
			\end{tikzpicture}
			\end{center}
		Here,
		$i$ is the injective group homomorphism in Proposition \ref{prop: group hom from ZG to Aut(Sig G)}.
	\end{thm}
	
	\begin{proof}
		First,
		we show that $Q$ is well-defined.
		Take $\Phi\in\Aut(\Sigma;G)$ and we show $\widetilde{\Phi}\in\Aut(G)_{\Sigma}$.
		We have the following commutative diagram:
			\begin{center}
			\begin{tikzpicture}[auto,scale=2]
				\node (Sig) at (0,0) {$\Sigma$};
				\node (G) at (2,0){$G$};
				\node (G') at (2,-1){$G$};
				\node (Sig') at (0,-1) {$\Sigma$};
				\node (torus) at (-2, 0) {$G^{(0)}\times\T$};
				\node (torus') at (-2, -1) {$G^{(0)}\times\T$};
				\draw[->] (Sig) to node {$\Phi$} (Sig') ;
				\draw[->] (Sig) to node {$q$} (G);
				\draw[->,swap] (Sig') to node {$q$} (G');
				\draw[->,swap] (torus) to node {$\Phi|_{G^{(0)}}\times \id_{\T}$} (torus');
				\draw[->] (torus) to node {$\iota$} (Sig);
				\draw[->,swap] (torus') to node {$\iota$} (Sig');
				\draw[->] (G) to node {$\widetilde{\Phi}$}(G');
			\end{tikzpicture}
			,
		\end{center}
		where $\iota$ and $q$ are the inclusion and quotient maps of $(G,\Sigma)$.
		Rewriting this diagram,
		we obtain the following one:
			\begin{center}
			\begin{tikzpicture}[auto,scale=2]
				\node (Sig) at (0,0) {$\Sigma$};
				\node (G) at (2,0){$G$};
				\node (G') at (2,-1){$G$};
				\node (Sig') at (0,-1) {$\Sigma$};
				\node (torus) at (-2, 0) {$G^{(0)}\times\T$};
				\node (torus') at (-2, -1) {$G^{(0)}\times\T$};
				\draw[->] (Sig) to node {$\Phi$} (Sig') ;
				\draw[->] (Sig) to node {$\widetilde{\Phi}\circ q$} (G);
				\draw[->,swap] (Sig') to node {$q$} (G');
				\draw[->,swap] (torus) to node {$\id_{G^{(0)}\times \T}$} (torus');
				\draw[->] (torus) to node {{\small $\iota\circ(\Phi^{-1}|_{G^{(0)}}\times\id_{\T})$}} (Sig);
				\draw[->,swap] (torus') to node {$\iota$} (Sig');
				\draw[->] (G) to node {$\id_G$}(G');
			\end{tikzpicture}
			.
		\end{center}
		Remark that the upper horizontal sequence is the twist $(G, \Sigma_{\widetilde{\Phi}})$.
		Hence this diagram shows that $\Sigma$ and $\Sigma_{\widetilde{\Phi}}$ are cohomologous.
		In particular, we obtain $\widetilde{\Phi}\in \Aut(G)_{\Sigma}$.
		
		Next, we show that $Q$ is surjective.
		Take $\Psi\in\Aut(G)_{\Sigma}$.
		Since $\Sigma$ and $\Sigma_{\Psi}$ are cohomologous,
		there exists an isomorphism $\Phi\colon\Sigma\to \Sigma_{\Psi}$ with the following commutative diagram:
		\begin{center}
			\begin{tikzpicture}[auto,scale=1.5]
				\node (Sig) at (0,0) {$\Sigma$};
				\node (G) at (2,-1){$G$};
				\node (Sig') at (0,-2) {$\Sigma$};
				\node (torus) at (-2, -1) {$G^{(0)}\times\T$};
				\draw[->] (Sig) to node {$\Phi$} (Sig') ;
				\draw[->] (Sig) to node {$q$} (G);
				\draw[->,swap] (Sig') to node {$\Psi\circ q$} (G);
				\draw[->] (torus) to node {$\iota$} (Sig);
				\draw[->,swap] (torus) to node {$\iota\circ (\Psi^{-1}|_{G^{(0)}})\times \id_{\T}$} (Sig');
			\end{tikzpicture}
			.
		\end{center}
		Rewriting this diagram,
		we obtain the following:
		\begin{center}
			\begin{tikzpicture}[auto,scale=2]
				\node (Sig) at (0,0) {$\Sigma$};
				\node (G) at (2,0){$G$};
				\node (G') at (2,-1){$G$};
				\node (Sig') at (0,-1) {$\Sigma$};
				\node (torus) at (-2, 0) {$G^{(0)}\times\T$};
				\node (torus') at (-2, -1) {$G^{(0)}\times\T$};
				\draw[->] (Sig) to node {$\Phi^{-1}$} (Sig') ;
				\draw[->] (Sig) to node {$q$} (G);
				\draw[->,swap] (Sig') to node {$q$} (G');
				\draw[->,swap] (torus) to node {$\Psi|_{G^{(0)}}\times\id_{\T}$} (torus');
				\draw[->] (torus) to node {{\small $\iota$}} (Sig);
				\draw[->,swap] (torus') to node {$\iota$} (Sig');
				\draw[->] (G) to node {$\Psi$}(G');
			\end{tikzpicture}
			.
		\end{center}
		This diagram shows that $\Phi^{-1}\in \Aut(\Sigma;G)$ and $\Psi=\widetilde{\Phi^{-1}}$.
		Hence $Q$ is surjective.
		
		It is easy to check $i(Z(G))\subset \ker Q$.
		We show the reverse inclusion and take $\Phi\in\ker Q$.
		Since we have $q(\Phi(\delta))=q(\delta)$ for $\delta\in\Sigma$ by $\widetilde{\Phi}=\id_G$,
		there uniquely exists $\widetilde{c}(\delta)\in\T$ with $\Phi(\delta)=\widetilde{c}(\delta)\delta$.
		It is straightforward to see that $\widetilde{c}\colon \Sigma\to \T$ is a homomorphism.
		Since we have $\widetilde{c}(\delta)=\Phi(\delta)\delta^{-1}$,
		$\widetilde{c}$ is continuous.
		We show that $\widetilde{c}$ factors through $q\colon \Sigma\to G$.
		It suffices to show $\widetilde{c}|_{G^{(0)}\times \T}=1$.
		Take $(x,z)\in G^{(0)}\times \T$.
		Then we have
		\[
		\widetilde{c}(x,z)zx=\widetilde{c}(x,z)(x,z)=\Phi(x,z)=z\Phi(x)=zx
		\]
		since $\Phi(x,z)=z\Phi(x)$ by $\Phi\in \Aut(\Sigma; G)$ and $\Phi(x)=\widetilde{\Phi}(x)=x$ by $\widetilde{\Phi}=\id_G$.
		Hence we obtain $\widetilde{c}(x,z)=1$.
		Therefore there exists $c\colon G\to \T$ with $\widetilde{c}=c\circ q$.
		Since the topology of $G$ coincides with the quotient topology of $q\colon \Sigma\to G$ by \cite[Lemma 2.7]{Armstrong_uniqueness_theorem_for_twistedgroupoid},
		$c$ is continuous and we obtain $c\in Z(G)$.
		Thus, we obtain $\Phi=\Phi_c$ and this completes the proof.
		\qed
	\end{proof}

	The following theorem is a special case of \cite[Theorem 5.10]{Jonarhan2025essentialCartan}.
	Indeed,
	the author of \cite{Jonarhan2025essentialCartan} treats not necessarily Hausdorff groupoids,
	while we only consider the Hausdorff case.
	
	\begin{thm}[{\cite[Theorem 5.10]{Jonarhan2025essentialCartan}}]\label{thm: J.Taylors theorem}
		
		Let $G$ be a locally compact Hausdorff effective \'etale groupoid and $\Sigma$ be a twist over $G$.
		Then $\Aut(C^*_r(\Sigma;G); C_0(G^{(0)}))$ is isomorphic to $\Aut(\Sigma;G)$.
		More precisely,
		for $\Phi\in \Aut(\Sigma;G)$,
		$\varphi_{\Phi}\in \Aut(C^*_r(\Sigma; G); C_0(G^{(0)}))$ is defined by
		\[
		\varphi_{\Phi}(f)(\delta)\defeq f(\Phi^{-1}(\delta))
		\]
		for $f\in C_c(\Sigma; G)$ and $\delta\in \Sigma$.
		Then 
		\[
		\varphi\colon \Aut(\Sigma; G)\ni \Phi\mapsto \varphi_{\Phi}\in \Aut(C^*_r(\Sigma; G); C_0(G^{(0)}))
		\]
		is a group isomorphism.
	\end{thm}

	\begin{cor}\label{cor: C*-analogue of Feldman moore}
		Let $G$ be a locally compact Hausdorff effective \'etale groupoid and $\Sigma$ be a twist over $G$.
		Put
		\[
		\mathfrak{W}_{C^*_r(\Sigma; G), C_0(G^{(0)})}\defeq \Aut(C^*_r(\Sigma; G); C_0(G^{(0)}))/\Aut_{C_0(G^{(0)})}(C^*_r(\Sigma;G)).
		\]
		Then $\mathfrak{W}_{C^*_r(\Sigma; G), C_0(G^{(0)})}$ is isomorphic to $\Aut(G)_{\Sigma}$ via an isomorphism $\widetilde{\varphi}$ with the following commutative diagram:
		\begin{center}
			\begin{tikzpicture}[auto,scale=2]
				\node (AutC*Gsig) at (0,-1) {$\Aut(C^*_r(\Sigma;G); C_0(G^{(0)}))$};
				\node (Weyl) at (2,-1){$\mathfrak{W}_{C^*_r(\Sigma; G), C_0(G^{(0)})}$};
				\node (AutGsig) at (2,0){$\Aut(G)_{\Sigma}$};
				\node (AutsigG) at (0,0) {$\Aut(\Sigma;G)$};
				\node (Autc0C*Gsig) at (-2.2, -1) {$\Aut_{C_0(G^{(0)})}(C^*_r(\Sigma;G))$};
				\node (ZG) at (-2.2, 0) {$Z(G)$};
				\draw[->] (AutC*Gsig) to node {$q$} (Weyl) ;
				\draw[->] (Autc0C*Gsig) to node {$\iota$} (AutC*Gsig);
				\draw[->] (AutsigG) to node {$Q$} (AutGsig);
				\draw[->] (ZG) to node {$i$} (AutsigG);
				\draw[<-,swap] (Weyl) to node {$\widetilde{\varphi}$} (AutGsig);
				\draw[<-,swap] (AutC*Gsig) to node {$\varphi$} (AutsigG);
				\draw[<-] (Autc0C*Gsig) to node {$\varphi\circ i$}(ZG);
			\end{tikzpicture}
			,
		\end{center}
		where $\varphi$ is the isomorphism in Theorem \ref{thm: J.Taylors theorem},
		$i$ and $Q$ are the homomorphisms in Theorem \ref{theorem: main theorem:short exact sequence},
		and $\iota$, $q$ are the inclusion and quotient maps respectively.
		
	\end{cor}
	
	\begin{proof}
		By Theorem \ref{theorem: main theorem:short exact sequence} and the definition of $\mathfrak{W}_{C^*_r(\Sigma; G), C_0(G^{(0)})}$,
		both horizontal short sequences are exact.
		In addition,
		$i$ and $\iota$ are injective, and $Q$ and $q$ are surjective.
		Hence it suffices to show 
		\[
		\varphi\circ i(Z(G))=\Aut_{C_0(G^{(0)})}(C^*_r(\Sigma;G))
		\]
		to conclude that $\tilde{\varphi}$ is an isomorphism by the five lemma.
		The inclusion
		\[
		\varphi\circ i(Z(G))\subset\Aut_{C_0(G^{(0)})}(C^*_r(\Sigma;G))
		\]
		follows from a straightforward calculation.
		To show the reverse inclusion,
		take $\psi\in \Aut_{C_0(G^{(0)})}(C^*_r(\Sigma; G))$.
		There exists $\Phi\in\Aut(\Sigma; G)$ with $\psi=\varphi_{\Phi}$ since $\varphi$ is surjective by Theorem \ref{thm: J.Taylors theorem}.
		To show $\psi=\varphi_{\Phi}\in \varphi\circ i(Z(G))$,
		it suffices to see $Q(\Phi)=\id_G$ by Theorem \ref{theorem: main theorem:short exact sequence},
		where $Q(\Phi)=\widetilde{\Phi}\in \Aut(G)_{\Sigma}$ is the automorphism in Theorem \ref{theorem: main theorem:short exact sequence}.
		Since we have $\varphi_{\Phi}(f)=f$ for all $f\in C_0(G^{(0)})$,
		it follows $\widetilde{\Phi}|_{G^{(0)}}=\id_{G^{(0)}}$.
		Since we assume that $G$ is effective and $\widetilde{\Phi}\in \Aut(G)$,
		we obtain $\widetilde{\Phi}=\id_G$ by Proposition \ref{prop: restriction of autG to G0 is injective}.
		This completes the proof.
		\qed
	\end{proof}

	\begin{rem}
		We remark that Corollary \ref{cor: C*-analogue of Feldman moore} can be viewed as a C*-algebraic analogue of \cite[Theorem 3]{Feldman_Moore_II},
		where the authors proved a corresponding statement in the von Neumann algebraic setting.
		If $\Sigma=\Sigma_{\sigma}$ arises from a normalized 2-cocycle $\sigma\in Z^2(G)$,
		then $\Aut(G)_{\Sigma}$ is isomorphic to $\Aut(G)_{\sigma}$ by Remark \ref{rem: aut(G)Sigma sigma coincides with cocycle preserving automorphisms}.
		In this case,
		the situation is closer to that considered in \cite[Theorem 3]{Feldman_Moore_II}.
		We also remark that,
		if $G$ is ample and $\sigma$-compact,
		then every twist over $G$ has a continuous section and hence arises from a 2-cocycle by \cite[Corollary 3.5]{BonickeK-theoryandhomotopies}.
		In addition,
		it is known that there exists a twist over a locally compact Hausdorff \'etale groupoid which does not arise from a normalized 2-cocycle (see \cite[Section 3]{ArmstrongNgAbraham_notrivial_twists}, for example).
		
		Finally,
		if a twist $\Sigma=G\times \T$ is trivial,
		then we obtain
		\[
		\Aut(C^*_r(G);C_0(G^{(0)}))\simeq \Aut(G)\ltimes Z(G),
		\]
		where the right hand side is the semidirect product with respect to the action $\Aut(G)\curvearrowright \Z(G)$ induced by the composition.
		This has already been proved in \cite[Corollary 3.2.2]{Komura2026*-hom}.
		We remark that the notation in \cite{Komura2026*-hom} is slightly different from that used in this paper.
		
	\end{rem}
	
	At the end of this section,
	we prove that $\varphi\circ i$ in Corollary \ref{cor: C*-analogue of Feldman moore} is an isomorphism as topological groups.
	Recall that the topology of $Z(G)$ is the compact-open topology,
	which coincides with the topology of uniform convergence on compact sets,
	while $\Aut_{C_0(G^{(0)})}(C^*_r(\Sigma;G))$ is equipped with the pointwise norm topology.
	
	\begin{prop}\label{prop: Z(G) is isom to Aut as top groups}
		The group isomorphism $\varphi\circ i\colon Z(G)\to	\Aut_{C_0(G^{(0)})}(C^*_r(\Sigma;G))$ in Corollary \ref{cor: C*-analogue of Feldman moore} is a homeomorphism and hence $\varphi\circ i$ is an isomorphism as topological groups.
	\end{prop}
	\begin{proof}
		Take a net $\{c_{\lambda}\}_{\lambda\in\Lambda}\subset Z(G)$ and $c\in Z(G)$.
		We put $\varphi_c\defeq \varphi(i(c))$ and $\varphi_{c_{\lambda}}\defeq\varphi(i(c_{\lambda}))$.
		Note that $\varphi_c\in \Aut_{C_0(G^{(0)})}(C^*_r(\Sigma;G))$ satisfies
		\[
		\varphi_c(f)(\delta)=c(q(\delta))f(\delta)
		\]
		for $f\in C_c(\Sigma;G)$ and $\delta\in\Sigma$,
		where $q\colon \Sigma\to G$ denotes the quotient map.
		To show that $\varphi\circ i$ is a homeomorphism,
		it suffices to show that $\{c_{\lambda}\}$ converges to $c$ if and only if $\{\varphi_{c_{\lambda}}\}_{\lambda\in \Lambda}$ converges to $\varphi_c$ since $\varphi\circ i$ is a bijection.
		
		First,
		we show that $\{\varphi_{c_{\lambda}}\}_{\lambda\in \Lambda}$ converges to $\varphi_c$ if $\{c_{\lambda}\}$ converges to $c$.
		It suffices to show $\{\varphi_{c_{\lambda}}(f)\}_{\lambda\in \Lambda}$ converges to $\varphi_c(f)$ in $C^*_r(\Sigma;G)$ for all $f\in C_c(\Sigma;G)$ such that $q(\insupp(f))\subset G$ is a bisection,
		since the linear span of such elements is dense in $C^*_r(\Sigma;G)$ by \cite[Lemma 2.13]{Armstrong_uniqueness_theorem_for_twistedgroupoid}.
		Since we have
		\[
		\insupp(\varphi_{c_{\lambda}}(f))=\insupp(\varphi_{c}(f))=\insupp(f)
		\]
		for all $\lambda\in\Lambda$ and $q(\insupp(f))$ is a bisection,
		$\varphi_{c_{\lambda}}(f)-\varphi_c(f)$ is a normalizer by Proposition \ref{prop: characterization of normalizer} and hence its reduced norm coincides with the supremum norm.
		Thus we obtain
		\begin{align*}
			\lVert \varphi_{c_{\lambda}}(f)-\varphi_{c}(f)\rVert&=\sup_{\delta\in \Sigma}\lvert c_{\lambda}(q(\delta))f(\delta)-c(q(\delta))f(\delta)\rvert \\
			&\leq \lVert f\rVert_{\infty} \sup_{\alpha\in q(\supp(f))}\lvert c_{\lambda}(\alpha)-c(\alpha)\rvert\to 0 &(\text{as $\lambda\to \infty$}),
		\end{align*}
		where $\lVert \cdot\rVert_{\infty}$ denotes the supremum norm.
		
		Next,
		we show the converse.
		Consider a compact set $K\subset G$.
		Using the compactness of $K$,
		we may take open bisections $V_i,U_i\subset G$ for $i=1,2,\dots n$ so that
		\begin{itemize}
			\item each $\overline{V_i}$ is compact and $\overline{V_i}\subset U_i$,
			\item for $i=1,2,\dots,n$,
			there exist continuous sections $\phi_i\colon U_i\to q^{-1}(U_i)$ of the quotient map $q|_{q^{-1}(U_i)}$ such that
			\[
			\T\times U_i\ni (z,\alpha)\mapsto z\phi_i(\alpha)\in q^{-1}(U_i)
			\]
			are homeomorphisms, and
			\item $K\subset \bigcup_{i=1}^n V_i$.
		\end{itemize}
		For each $i=1,2,\cdots,n$,
		take $f_i\in C_c(U_i)$ such that $f_i|_{\overline{V_i}}=1$ and $0\leq f_i\leq 1$.
		Define $\widetilde{f_i}\in C_c(\Sigma;G)$ by
		\begin{align*}
			\widetilde{f_i}(\delta)\defeq
			\begin{cases}
				zf_i(\alpha) & (\delta\in q^{-1}(U_i)),\\
				0 & (\text{otherwise}),
			\end{cases}
		\end{align*}
		where $(z,\alpha)\in \T\times U_i$ is the unique element satisfying $\delta=z\phi_i(\alpha)$.
		Since the open support of $\varphi_{c_{\lambda}}(\widetilde{f_i})-\varphi_c(f_i)$ is a bisection,
		its reduced norm coincides with the supremum norm.
		Hence,
		\begin{align*}
			\lVert \varphi_{c_{\lambda}}(\widetilde{f_i})-\varphi_c(\widetilde{f_i})\rVert&=\sup_{\delta\in \Sigma}\lvert \varphi_{c_{\lambda}}(\widetilde{f_i})(\delta)-\varphi_c(\widetilde{f_i})(\delta)\rvert \\
			&=\sup_{(z,\alpha)\in \T\times U_i}\lvert c_{\lambda}(\alpha)zf_i(\alpha)-c(\alpha)zf_i(\alpha)\rvert \\
			&\geq \sup_{\alpha\in \overline{V_i}}\lvert c_{\lambda}(\alpha)-c(\alpha)\rvert.
		\end{align*}
		In addition,
		since we have $K\subset \bigcup_{i=1}^n \overline{V_i}$ and the above inequality,
		we obtain
		\begin{align*}
		\sup_{\alpha\in K}\lvert c_{\lambda}(\alpha)-c(\alpha)\rvert&\leq \sup_{i=1,\dots,n}\sup_{\alpha\in \overline{V_i}}\lvert c_{\lambda}(\alpha)-c(\alpha)\rvert\\
		&\leq \sup_{i=1,\dots,n}\lVert\varphi_{c_{\lambda}}(\widetilde{f_i})-\varphi_c(\widetilde{f_i})\rVert.
		\end{align*}
		Since we assume that $\{\varphi_{c_{\lambda}}\}_{\lambda\in\Lambda}$ converges to $\varphi$ in the pointwise norm topology,
		$\{c_{\lambda}\}_{\lambda\in \Lambda}$ converges to $c$ uniformly on $K$.
		Hence $\{c_{\lambda}\}_{\lambda\in\Lambda}$ converges to $c$ in $Z(G)$.
		\qed
	\end{proof}

	\section{Compact abelian subgroups of 1-cocycle groups $Z(G)$}
	
	From Corollary \ref{cor: C*-analogue of Feldman moore},
	it follows that the group of Cartan-fixing automorphisms is isomorphic to $Z(G)$ and hence abelian.
	In this section,
	we investigate compact abelian subgroups of the group of Cartan-fixing automorphisms.
	First, we establish a realization theorem for compact abelian groups (Theorem \ref{Theorem: existence of universal Cartan subalgebra of Kirchberg algebra}).
	Next, we prove an obstruction theorem for expansive groupoids (Theorem \ref{theorem every compact subgroup of Cartan fixing are fin gen}).

	\subsection{Realization of compact abelian groups}
	
	Our purpose in this subsection is to prove Theorem \ref{Theorem: existence of universal Cartan subalgebra of Kirchberg algebra}.
	Our strategy is as follows.
	First, we show Theorem \ref{Theorem: existence of universal Cartan subalgebra of Kirchberg algebra} for the Cuntz algebra $\mathcal{O}_{\infty}$.
	Then we apply $\mathcal{O}_{\infty}$-absorption theorem and obtain Theorem \ref{Theorem: existence of universal Cartan subalgebra of Kirchberg algebra} for the general Kirchberg case.
	
	\subsubsection{The case of the Cuntz algebra $\mathcal{O}_{\infty}$} \label{subsubsection:Cuntz algebra case}
	
	The Cuntz algebra $\mathcal{O}_{\infty}$ is the universal unital C*-algebra generated by $\{S_i\}_{i=1}^{\infty}$ such that
	\[
	S_i^*S_j=
	\begin{cases}
		1 & (i=j),\\
		0 & (i\not=j)
	\end{cases}
	\]
	for $i,j\in\N\setminus\{0\}$.
	Put $\Sigma\defeq \N\setminus\{0\}$ and $\Sigma^*\defeq \bigcup_{n\in\N}\Sigma^n$,
	which is the set of all finite sequences on $\Sigma$.
	Note that $\Sigma^0$ is a singleton whose unique element $\varepsilon$ is called the empty word.
	When $\mu\in\Sigma^k$,
	we define the length of $\mu$ as $\lvert \mu\rvert \defeq k$.
	For $\mu\in\Sigma^k$,
	put $S_{\mu}\defeq S_{\mu_1}S_{\mu_2}\cdots S_{\mu_k}\in\mathcal{O}_{\infty}$.
	Then we have
	\[
	\mathcal{O}_{\infty}=\overline{\Span\{S_{\mu}S_{\nu}^*\mid \mu,\nu\in \Sigma^*\}}.
	\]
	We put
	\[
	D_{\infty}\defeq \overline{\Span\{S_{\mu}S_{\mu}^*\mid \mu\in\Sigma^*\}}.
	\]
	Then $D_{\infty}$ is a commutative subalgebra of $\mathcal{O}_{\infty}$.
	Moreover,
	$D_{\infty}\subset \mathcal{O}_{\infty}$ is a Cartan subalgebra since this inclusion has a groupoid model which is effective and \'etale.
	We next recall this groupoid model of $\mathcal{O}_{\infty}$.
	See \cite{Paterson2002} or \cite[Example 2.2.7]{Komurasubmod} for details.
	
	We keep the above notation $\Sigma\defeq \N\setminus\{0\}$ and $\Sigma^*\defeq \bigcup_{n\in\N}\Sigma^n$.
	Let $P_{\infty}$ denote the Polycyclic monoid of infinite degree.
	Namely,
	$P_{\infty}$ is the universal inverse semigroup defined by
	\[
	P_{\infty}\defeq \i<\{s_i\}_{i=1}^{\infty},0,1 \mid s_i^*s_j=\delta_{i,j}1>.
	\]
	Remark that
	\[
	P_{\infty}=\{s_{\mu}s_{\nu}^*\mid\mu,\nu\in\Sigma^*\}\cup\{0\},
	\]
	where $s_{\mu}\defeq s_{\mu_1}s_{\mu_2}\cdots s_{\mu_k}$ for $\mu\in\Sigma^*$ with $\lvert\mu\rvert=k$.
	Let $X\defeq \Sigma^*\cup\Sigma^{\N}$ be the set of all finite or infinite sequences on $\Sigma$.
	For $\mu\in\Sigma^*$ and a finite set $F\subset \Sigma^*$,
	define $C_{\mu,F}\subset X$ to be the set of all sequences which begin with $\mu$ and do not begin with the elements in $F$.
	If $F=\emptyset$,
	we simply write $C_{\mu}\defeq C_{\mu,\emptyset}$.
	Then a family of all $C_{\mu,F}$ forms an open basis of $X$ and $X$ is a compact Hausdorff space with respect to the topology generated by all $C_{\mu, F}$.
	Note that $C_{\mu,F}$ is compact open in $X$.
	In addition,
	$X$ is a Cantor space,
	that is,
	second countable,
	totally disconnected,
	compact Hausdorff space without isolated points.
	For $\mu,\nu\in\Sigma^*$,
	define $\alpha_{s_{\mu}s_{\nu}^*}\colon C_{\nu}\to C_{\mu}$ by
	\[
	\alpha_{s_{\mu}s_{\nu}^*}(\nu x)=\mu x
	\]
	for all $x\in X$.
	Then we obtain the action $\alpha\colon P_{\infty}\curvearrowright X$.
	Put $G_{\infty}\defeq P_{\infty}\ltimes_{\alpha}X$.
	Then $G_{\infty}$ is a locally compact Hausdorff second countable effective \'etale groupoid.
	In addition,
	$G_{\infty}$ is minimal and locally contracting since we will observe that $G_{\infty}$ has a minimal locally contracting subgroupoid which contains $G_{\infty}^{(0)}$ in Lemma \ref{lem: cocycle on Cuntz groupoid} and Lemma \ref{lemma: locally contractiveness of kersigma}.
	
	The inclusion $D_{\infty}\subset \mathcal{O}_{\infty}$ is isomorphic to $C(G_{\infty}^{(0)})\subset C^*_r(G_{\infty})$.
	Indeed,
	there exists a *-isomorphism $\varphi\colon C^*_r(G_{\infty})\to \mathcal{O}_{\infty}$ such that
	\[
	\varphi(\chi_{[s_{\mu}s_{\nu}^*, C_{\nu}]})=S_{\mu}S_{\nu}^*,
	\]
	where $\chi_{[s_{\mu}s_{\nu}^*, C_{\nu}]}\in C_c(G_{\infty})$ denotes the characteristic function on the compact open set $[s_{\mu}s_{\nu}^*, C_{\nu}]\subset G_{\infty}$.
	Note that $\varphi$ maps $C(G_{\infty}^{(0)})$ to $D_{\infty}$.
	In particular,
	by \cite[Theorem 5.6.18]{brown},
	$G_{\infty}$ is amenable since $\mathcal{O}_{\infty}$ is nuclear.
	
	\begin{lem}\label{lem: cocycle on Cuntz groupoid}
		Let $G_{\infty}$ be the above groupoid and $\Gamma$ be a countable discrete group.
		Then there exists a surjective continuous 1-cocycle $\sigma\colon G_{\infty}\to \Gamma$ such that $\ker\sigma$ is minimal.
	\end{lem}
	
	\begin{proof}
		Fix a (possibly non-injective) surjective map $t\colon \N\setminus\{0\}\to \Gamma$ and write $t_i\defeq t(i)$ for $i\in\N\setminus\{0\}$. 
		Consider the 1-cocycle $\sigma\colon G_{\infty}\to \Gamma$ induced by a partial homomorphism $P_{\infty}\setminus \{0\}\to \Gamma$ which maps $s_i$ to $t_i$ for $i\in\N$ (see the end of Subsection \ref{subsection: inverse semigroup actions} for 1-cocycles induced by partial homomorphisms).
		Namely,
		$\sigma$ satisfies
		\[
		\sigma([s_{\mu}s_{\nu}^*, x])=t_{\mu}t_{\nu}^*,
		\]
		where $\mu,\nu\in \Sigma^*$,
		$x\in C_{\nu}$ and $t_{\mu}\defeq t_{\mu_1}t_{\mu_2}\cdots t_{\mu_{\lvert\mu\rvert}}$.
		Then $\sigma$ is surjective since $\sigma([s_i, \varepsilon])=t_i$ for $i\in\N\setminus\{0\}$.
		To check that $\ker\sigma$ is minimal,
		consider a non-empty open set $U\subset X$ and $x\in X$.
		We may assume that $U=C_{\mu, \mu F}$ for some $\mu\in\Sigma^*$ and a finite set $F\subset \Sigma^*\setminus\{\varepsilon\}$. 
		One may choose $j\in\Sigma$ different from the first letter of every word in $F$ since $F$ is finite and $\Sigma$ is countably infinite.
		Then we have
		\[C_{\mu j}\cap \bigg(\bigcup_{\nu\in F}C_{\mu\nu}\bigg)=\emptyset.\]		
		Since $t$ is surjective,
		there exists $i\in\N\setminus\{0\}$ such that $t_i=t_{\mu j}^{-1}$.
		Put
		\[
		\alpha\defeq [s_{\mu j i}, x] \in G_{\infty}.
		\]
		Then one may check $\alpha\in \ker \sigma$, $r(\alpha)\in C_{\mu, \mu F}$ and $d(\alpha)=x$.
		Hence we have shown that $\ker \sigma$ is minimal.
		\qed
	\end{proof}
		
	\begin{lem}\label{lemma: locally contractiveness of kersigma}
		Keep the notations in Lemma \ref{lem: cocycle on Cuntz groupoid}.
		Then $\ker\sigma$ is locally contracting.
	\end{lem}
	
	\begin{proof}
		
		Take a nonempty open set $U\subset X$.
		There exist $\mu\in\Sigma^*$ and a finite set $F\subset \Sigma^*\setminus\{\varepsilon\}$ such that $C_{\mu,\mu F}\subset U$.
		Take $j\in \Sigma$ so that $j$ is different from the first letter of every word in $F$.
		Then we obtain $C_{\mu j}\subsetneq C_{\mu,\mu F}$ from the choice of $j$ and $\mu\in C_{\mu,\mu F}\setminus C_{\mu j}$.
		Since $t\colon \N\setminus\{0\}\to \Gamma$ is surjective,
		there exists $i\in \N\setminus\{0\}$ such that $t_i=t_j^{-1}$.
		Then we have $C_{\mu ji}\subsetneq C_{\mu,\mu F}$.
		Put
		\[
		V=C_{\mu,\mu F} \quad\text{and}\quad
		S\defeq [s_{\mu ji }s_{\mu}^*, C_{\mu, \mu F}].
		\]
		Then $V$ is a nonempty clopen subset of $X$ and $S\subset \ker\sigma$ is a compact open bisection such that
		\[
		\overline{V}=V=d(S)
		\]
		and
		\[
		r(S\overline{V})=C_{\mu ji}\subsetneq C_{\mu,\mu F}=V.
		\]
		Hence $\ker\sigma$ is locally contracting.
		\qed
	\end{proof}
	
	\begin{thm}\label{thm: main theorem for Cuntz algebra}
		Let $H$ be a second countable compact abelian group.
		Then there exists an injective continuous homomorphism $\tau\colon H\hookrightarrow \Aut_{D_{\infty}}(\mathcal{O}_{\infty})$ such that 
		\begin{enumerate}
			\item the fixed point subalgebra $\mathcal{O}_{\infty}^{\tau}$ is a unital Kirchberg algebra satisfying the UCT,
			
			\item$D_{\infty}\subset \mathcal{O}_{\infty}^{\tau}$ and
			\item$\tau(H)=\Aut_{\mathcal{O}_{\infty}^{\tau}}(\mathcal{O}_{\infty})$.
		\end{enumerate}
		In particular,
		$H$ is isomorphic to $\Aut_{\mathcal{O}_{\infty}^{\tau}}(\mathcal{O}_{\infty})$ via $\tau$.
	\end{thm}
	
	\begin{proof}
		We may identify the inclusion $D_{\infty}\subset \mathcal{O}_{\infty}$ with $C(G_{\infty}^{(0)})\subset C^*_r(G_{\infty})$,
		where $G_{\infty}$ is the standard groupoid model of $\mathcal{O}_{\infty}$ stated above Lemma \ref{lem: cocycle on Cuntz groupoid}.
		Let $\Gamma\defeq \widehat{H}$ denote the Pontryagin dual group of $H$.
		Since $H$ is second countable,
		$\Gamma$ is countable.
		By Lemma \ref{lem: cocycle on Cuntz groupoid} and \ref{lemma: locally contractiveness of kersigma},
		there exists a surjective 1-cocycle $\sigma \colon G_{\infty}\to \Gamma$ such that $\ker\sigma$ is minimal and locally contracting.
		Since $G^{(0)}_{\infty}\subset \ker\sigma \subset G^{(0)}_{\infty}$ and $G^{(0)}_{\infty}$ is effective,
		$\ker\sigma$ is also effective.
		Put $B=C^*_r(\ker\sigma)$.
		Since $G_{\infty}$ is amenable,
		its open subgroupoid $\ker \sigma$ is also amenable by \cite[Proposition 9.77]{Williamstoolkit}.
		Hence $B$ is nuclear\footnote{In this situation,
		one can check that $B$ is nuclear more directly. Indeed,
		since $\ker\sigma$ is clopen in $G$,
		there exists a conditional expectation from $\mathcal{O}_{\infty}$ to $B$ by \cite[Lemma 3.4]{BrownExelFuller2021}.
		Hence $B$ is nuclear by \cite[Proposition 10.1.2]{brown}.} and satisfies the UCT by \cite[Theorem 5.6.18]{brown} and \cite[Theorem 3.1]{BarlakLiCartanandUCTI}. 
		By \cite[Definition 3.3.1, Proposition 3.3.2]{Komura2025Weyl},
		$\sigma \colon G_{\infty}\to \Gamma$ induces the action $\widehat{\sigma}\colon H\curvearrowright C^*_r(G_{\infty})$ and $B=C^*_r(G_{\infty})^{\widehat{\sigma}}$ holds.
		Since $\sigma$ is surjective,
		$\widehat{\sigma}$ is faithful by \cite[Proposition 3.3.3]{Komura2025Weyl} and hence $\widehat{\sigma}\colon H\to \Aut_{C(G_{\infty}^{(0)})}(C^*_r(G_{\infty}))$ is injective.
		Since $\ker\sigma$ is effective, minimal and locally contracting,
		$B$ is a simple purely infinite C*-algebra by \cite[Theorem 5.1]{Brown2014} and Theorem \ref{thm: locally contracting implies purely infinite}.
		Hence $B$ is a unital Kirchberg algebra with $C(G_{\infty}^{(0)})\subset B\subset C^*_r(G_{\infty})$.
		Finally,
		by \cite[Corollary 3.3.5]{Komura2025Weyl},
		we have $\widehat{\sigma}(H)=\Aut_{C^*_r(G_{\infty})^{\widehat{\sigma}}}(C^*_r(G_{\infty}))$.
		Putting $\tau\defeq \widehat{\sigma}$,
		we have completed the proof.
		\qed
	\end{proof}

	\subsubsection{The general Kirchberg case}
	
	We first recall several results concerning tensor absorption,
	which will be used in the proof of Theorem \ref{Theorem: existence of universal Cartan subalgebra of Kirchberg algebra}.

	\begin{thm}[{\cite[Theorem 4]{Archbold1975},\cite[Theorem 4]{Batty1976}}] \label{thm: commutant of tensor product}
		Let $A, B$ be unital C*-algebras.
		Then the relative commutant of $A\otimes 1_B$ in $A\otimes B$ coincides with $Z(A)\otimes B$,
		where $Z(A)$ denotes the center of $A$ and the tensor product is the minimal tensor product.
	\end{thm}
	
	\begin{prop}\label{prop: isom on tensor and tensor factor}
		Let $A$ and $B$ be unital C*-algebras.
		Assume that $A$ is nonzero and has the trivial center $Z(A)=\C 1_A$.
		Then the map
		\[
		\Phi\colon\Aut(B)\ni \psi\mapsto \id_A\otimes \psi \in\Aut_{A\otimes 1_B}(A\otimes B)
		\]
		is a group isomorphism.
	\end{prop}
	
	\begin{proof}
		It is straightforward to check that $\Phi$ is a group homomorphism.
		Since we assume that $A$ is nonzero,
		one can see that $\Phi$ is injective.
		We show that $\Phi$ is surjective and take $\varphi\in\Aut_{A\otimes 1_B} (A\otimes B)$.
		For $b\in B$,
		we claim $\varphi(1\otimes b)\in 1_A\otimes B$.
		Indeed,
		for $a\in A$,
		we have $a\otimes 1=\varphi(a\otimes 1)$ and
		\begin{align*}
			\varphi(1\otimes b)(a\otimes 1)=\varphi(1\otimes b)\varphi(a\otimes 1)=\varphi(a\otimes b)=(a\otimes 1)\varphi(1\otimes b).
		\end{align*}
		Hence $\varphi(1\otimes b)$ belongs to the relative commutant of $A\otimes 1_B$ in $A\otimes B$.
		By Theorem \ref{thm: commutant of tensor product},
		$\varphi(1\otimes b)\in 1_A\otimes B$ and there uniquely exists $\psi(b)\in B$ such that $\varphi(1\otimes b)=1\otimes \psi(b)$.
		Since $\varphi^{-1}$ also preserves $A\otimes 1_B$,
		it turns out that $\psi\in \Aut(B)$.
		Since $\varphi=\id_A\otimes \psi$,
		we have shown that the map $\Phi$ is surjective and hence an isomorphism.
		\qed
	\end{proof}

	\begin{prop}\label{prop: isom between groups of subalgebra preserving automorphisms}
		Let $A$ be a unital nonzero C*-algebra with the trivial center $Z(A)=\C 1_A$.
		Assume that $C\subset B$ is an inclusion of C*-algebras with $1_B\in C$.
		Let
		\[
		\Phi\colon \Aut(B)\to \Aut_{A\otimes 1_B}(A\otimes B)
		\]
		denote the isomorphism in Proposition \ref{prop: isom on tensor and tensor factor}.
		Then $\Phi(\Aut_C(B))=\Aut_{A\otimes C}(A\otimes B)$.
		In particular,
		$\Aut_C (B)$ is isomorphic to $\Aut_{A\otimes C}(A\otimes B)$ via $\Phi$.
	\end{prop}
	
	\begin{proof}
		For $\psi\in \Aut(B)$,
		one can see that $\psi\in\Aut_C(B)$ if and only if $\id_A\otimes \psi\in\Aut_{A\otimes C}(A\otimes B)$.
		The proposition follows from this fact.
		\qed
	\end{proof}
	
	Now,
	we are ready to show the main realization theorem of this subsection.
	
	\begin{thm}\label{Theorem: existence of universal Cartan subalgebra of Kirchberg algebra}
		Let $A$ be a unital Kirchberg algebra satisfying the UCT.
		Then there exists a Cartan subalgebra $D\subset A$ whose Gelfand spectrum is a Cantor space with the following property:
		for any second countable compact abelian group $H$,
		there exists a continuous injective homomorphism $\tau\colon H\hookrightarrow  \Aut_D(A)$ such that the following conditions hold:
		
		\begin{enumerate}
			\item the fixed point algebra $A^{\tau}$ is a unital Kirchberg algebra satisfying the UCT,
			\item $D\subset A^{\tau}$,
			\item $\tau(H)=\Aut_{A^\tau}(A)$.
		\end{enumerate} 
		In particular,
		$H$ is isomorphic to $\Aut_{A^\tau}(A)$ via $\tau$.
	\end{thm}
	
	\begin{proof}
		Fix an isomorphism $\varphi\colon A\otimes \mathcal{O}_{\infty}\to A$ by using Kirchberg's absorption theorem (\cite[Theorem 3.14]{Kirchberg_Phillips_Embedding}, \cite[Theorem 7.2.6 (ii)]{Rordamclassification}).
		By \cite[Lemma 5.6]{XinLiRenaultCartanexistence},
		which relies on \cite{SpielberggraphBasedmodel},
		$A$ has a Cartan subalgebra $\widetilde{D}\subset A$ whose Gelfand spectrum is a Cantor space.
		Put $D\defeq \varphi(\widetilde{D}\otimes D_{\infty})\subset A$.
		Note that the Gelfand spectrum of $D$ is a Cantor space since the same holds for both $\widetilde{D}$ and $D_{\infty}$.
		Since the tensor product of Cartan subalgebras is also a Cartan subalgebra by \cite[Lemma 5.1]{BarlakLiCartanandUCTI}, 
		$D\subset A$ is a Cartan subalgebra.
		Now,
		to show that $D\subset A$ satisfies the properties in the statement,
		it suffices to show that $\widetilde{D}\otimes D_{\infty}\subset A\otimes \mathcal{O}_{\infty}$ satisfies the same properties.

		Let $H$ be a second countable compact abelian group.
		By Theorem \ref{thm: main theorem for Cuntz algebra},
		there exists a continuous injective homomorphism $\widetilde{\tau}\colon H\hookrightarrow \Aut_{D_{\infty}}(\mathcal{O}_{\infty})$ such that $\mathcal{O}_{\infty}^{\widetilde{\tau}}$ is a Kirchberg algebra and $\widetilde{\tau}(H)=\Aut_{\mathcal{O}_{\infty}^{\widetilde{\tau}}}(\mathcal{O}_{\infty})$.
		For $h\in H$,
		put $\tau'_h=\id_{A}\otimes \widetilde{\tau}_h\in \Aut(A\otimes \mathcal{O}_{\infty})$.
		Then this induces the continuous injective homomorphism 
		\[\tau'\colon H\hookrightarrow \Aut_{A\otimes D_{\infty}}(A\otimes \mathcal{O}_{\infty})(\subset \Aut_{\widetilde{D}\otimes D_{\infty}}(A\otimes \mathcal{O}_{\infty})).\]
		Since $A$ is simple and hence has the trivial center,
		we have the isomorphism
		\[
		\Phi\colon \Aut_{\mathcal{O}_{\infty}^{\widetilde{\tau}}}(\mathcal{O}_{\infty})\to\Aut_{A\otimes \mathcal{O}_{\infty}^{\widetilde{\tau}}}(A\otimes \mathcal{O}_{\infty}).
		\]
		in Proposition \ref{prop: isom between groups of subalgebra preserving automorphisms}.
		Since we have $\tau'=\Phi\circ\widetilde{\tau}$,
		we obtain
		\[
		\tau'(H)=\Aut_{A\otimes \mathcal{O}_{\infty}^{\widetilde{\tau}}}(A\otimes \mathcal{O}_{\infty}).
		\]
		In addition,
		we have
		\[A\otimes \mathcal{O}_{\infty}^{\tilde{\tau}}=(A\otimes \mathcal{O}_{\infty})^{\tau'}.\]
		Since $A$ and $\mathcal{O}_{\infty}$ are Kirchberg algebras satisfying the UCT,
		$A\otimes \mathcal{O}_{\infty}^{\widetilde{\tau}}$ is also a Kirchberg algebra by \cite[Proposition 4.1.8, 8.4.12]{Rordamclassification}.
		Hence,
		we have proved that the Cartan subalgebra $\widetilde{D}\otimes D_{\infty}\subset A\otimes \mathcal{O}_{\infty}$ satisfies the properties in the statement.
		Now,
		via the isomorphism $\varphi\colon A\otimes \mathcal{O}_{\infty}\to A$,
		we have completed the proof.
		\qed
		
		\end{proof}

	\subsection{Non-existence of compact abelian groups}
	
	In the previous subsection,
	we established a realization result for compact abelian groups as groups of Cartan-fixing automorphisms.
	In this subsection,
	we show that there are obstructions to such realization.
	First,
	we introduce the notion of expansive groupoids.
	Recall that $\Bis^c(G)$ denotes the inverse semigroup of compact open bisections of an \'etale groupoid $G$.

	\begin{defi}[{\cite[Definition 5.3]{NEKRASHEVYCH_2019}, \cite[Definition 2.1.10]{Komura2025Weyl}}]
		Let $G$ be a locally compact Hausdorff \'etale groupoid.
		Then $G$ is said to be expansive if there exists a finite set $F\subset \Bis^c(G)$ such that the inverse semigroup generated by $F$ in $\Bis^c(G)$ forms an open basis of $G$.
	\end{defi}
	
	\begin{rem}
		While our definition of expansiveness is phrased differently,
		it is equivalent to that in \cite[Definition 5.3]{NEKRASHEVYCH_2019}.
		See \cite[Remark 2.1.11]{Komura2025Weyl} for further discussion.
		Note that an expansive groupoid is ample.
	\end{rem}
	
	Expansive groupoids often arise from finitely generated inverse semigroup actions.
	
	\begin{ex}[{\cite[Example 2.1.12]{Komura2025Weyl}}]\label{ex: finitely generated inverse semigroup actions yields expansive groupoids}
		Let $S$ be a finitely generated inverse semigroup,
		$X$ be a locally compact Hausdorff space and $\sigma\colon S\curvearrowright X$ be an action such that $S\ltimes_{\sigma}X$ is Hausdorff.
		Assume that $\dom(\sigma_e)\subset X$ is a compact open set for each $e\in E(S)$ and $\{\dom(\sigma_e)\}_{e\in E(S)}$ is a basis of $X$.
		Then $S\ltimes_{\sigma}X$ is expansive.
		In particular,
		the standard groupoid models of the Cuntz algebras $\mathcal{O}_n$ are expansive for $n\in\N$ with $n\geq 2$ (see \cite[Example 2.1.12, 4.1.1]{Komura2025Weyl} for these facts).
		In contrast,
		the standard groupoid model of $\mathcal{O}_{\infty}$ described in Subsubsection \ref{subsubsection:Cuntz algebra case} is not expansive.
		We will observe this in Example \ref{example: groupoid model of o_infinity is not expansive}.
		
	\end{ex}
	
	The following simple observation is the key to the proof of Theorem \ref{theorem every compact subgroup of Cartan fixing are fin gen}.
	
	\begin{prop}\label{prop: image of expansive groupoid is finitely generated}
		Let $G$ be a locally compact Hausdorff \'etale groupoid,
		$\Gamma$ be a discrete group and $c\colon G\to\Gamma$ be a continuous 1-cocycle.
		If $G$ is expansive,
		then the group generated by $c(G)$ in $\Gamma$ is finitely generated.
		In particular,
		if $c$ is surjective,
		then $\Gamma$ is finitely generated.
	\end{prop}
	
	\begin{proof}
		Take a finite set $F\subset \Bis^c(G)$ so that the inverse semigroup generated by $F$ forms a basis of $G$.
		For $S\in F$,
		there exists a finite set $A_F\subset \Gamma$ such that
		\[
		S=\coprod_{t\in A_F}(S\cap c^{-1}(t)).
		\]
		By replacing $F$ with $\{S\cap c^{-1}(t)\mid S\in F, t\in A_F\}$,
		we may assume that $c(S)$ is a singleton for each $S\in F$.
		Take $t_S\in \Gamma$ so that $c(S)=\{t_S\}$ for $S\in F$.
		Then the group generated by $c(G)$ coincides with the group generated by $\{t_S\}_{S\in F}$.
		Hence the group generated by $c(G)$ is finitely generated.
		\qed
	\end{proof}
	
	\begin{ex}\label{example: groupoid model of o_infinity is not expansive}
		We observe that the standard groupoid model $G_{\infty}\defeq P_{\infty}\ltimes_{\alpha} X$ of the Cuntz algebra $\mathcal{O}_{\infty}$ described in Subsubsection \ref{subsubsection:Cuntz algebra case} is not expansive.
		Indeed,
		let $\Gamma$ be a countable group that is not finitely generated (for example,
		$\Gamma=\bigoplus_{\N}\Z$).
		Then there exists a surjective continuous 1-cocycle $c\colon G_{\infty}\to \Gamma$ by Lemma \ref{lem: cocycle on Cuntz groupoid}.
		Since $\Gamma$ is not finitely generated,
		$G$ is not expansive by Proposition \ref{prop: image of expansive groupoid is finitely generated}.
	\end{ex}
	
	The following is a main theorem of this subsection.
	
	\begin{thm}\label{theorem every compact subgroup of Cartan fixing are fin gen}
		Let $G$ be a locally compact Hausdorff expansive effective \'etale groupoid.
		Assume that $H\subset \Aut_{C_0(G^{(0)})}(C^*_r(G))$ is compact and the fixed point subalgebra $C^*_r(G)^H$ is prime.
		Then the Pontryagin dual group $\widehat{H}$ is finitely generated.
	\end{thm}
	\begin{proof}
		Let $\tau\colon H\curvearrowright C^*_r(G)$ denote the canonical action.
		By \cite[Proposition 3.3.4]{Komura2025Weyl},
		there exists a continuous 1-cocycle $c\colon G\to\widehat{H}$ such that
		\[
		\tau_h(a)(\alpha)=c(\alpha)(h)a(\alpha)
		\]
		for all $h\in H$, $a\in C^*_r(G)$ and $\alpha\in G$.
		By \cite[Proposition 3.3.2]{Komura2025Weyl},
		we have $C^*_r(G)^{\tau}=C^*_r(\ker c)$.
		Since we assume that $C^*_r(\ker c)=C^*_r(G)^{\tau}$ is prime,
		$\ker c$ is topologically transitive by Corollary \ref{cor: prime is equivalent to topologically transitive}.
		Hence,
		by \cite[Proposition 3.3.3]{Komura2025Weyl},
		$c$ is surjective since $\tau$ is a faithful action.
		Therefore,
		$\widehat{H}$ is finitely generated by Proposition \ref{prop: image of expansive groupoid is finitely generated}.
		\qed
		\end{proof}

	We present a non-uniqueness result of Cartan subalgebras as an application of Theorem \ref{theorem every compact subgroup of Cartan fixing are fin gen}.
	We remark that the non-uniqueness of Cartan subalgebras itself is not new.
	Indeed,
	in \cite[Proposition 5.7]{XinLiRenaultCartanexistence},
	the authors showed that every unital Kirchberg algebra satisfying the UCT admits infinitely many pairwise inequivalent Cartan subalgebras whose Gelfand spectra are all homeomorphic to a Cantor space.
	Moreover,
	they showed that the analogous statement holds for stable Kirchberg algebras as well in \cite[Proposition 5.8]{XinLiRenaultCartanexistence}.
	However, our approach is different from theirs.
	In \cite[Proposition 5.7]{XinLiRenaultCartanexistence},
	the authors distinguished Cartan subalgebras by examining the ranks of isotropy groups arising from their groupoid models.
	In contrast, we distinguish Cartan subalgebras using groups of Cartan-fixing automorphisms.
	
	\begin{cor}\label{cor: inequivalent Cartan}
		Let $G$ be a locally compact Hausdorff expansive effective \'etale groupoid with the compact unit space $G^{(0)}$.
		Assume that $C^*_r(G)$ is a Kirchberg algebra.
		Then $C^*_r(G)$ admits a Cartan subalgebra $D\subset C^*_r(G)$ such that
		\begin{itemize}
			\item the Gelfand spectrum of $D$ is a Cantor space, and
			\item $D$ is not equivalent to $C(G^{(0)})$,
			in the sense that there is no $\varphi\in \Aut (C^*_r(G))$ satisfying $\varphi(D)=C(G^{(0)})$.
		\end{itemize}
	\end{cor}
	
	\begin{proof}
		Note that $C^*_r(G)$ satisfies the UCT by \cite[Theorem 3.1]{BarlakLiCartanandUCTI}.
		By Theorem \ref{Theorem: existence of universal Cartan subalgebra of Kirchberg algebra},
		there exists a Cartan subalgebra $D\subset C^*_r(G)$ with the properties in Theorem \ref{Theorem: existence of universal Cartan subalgebra of Kirchberg algebra}.
		In particular,
		the Gelfand spectrum of $D$ is a Cantor space.
		Let $\mathcal{E}$ and $\mathcal{F}$ denote the sets of all compact abelian subgroups of $\Aut_{C(G^{(0)})}(C^*_r(G))$ (resp.\ $\Aut_{D}(C^*_r(G))$) such that the fixed point algebras are prime.
		Then $\mathcal{F}$ contains all second countable compact abelian groups by Theorem \ref{Theorem: existence of universal Cartan subalgebra of Kirchberg algebra},
		whereas,
		by Theorem \ref{theorem every compact subgroup of Cartan fixing are fin gen}, $\mathcal{E}$ does not contain compact abelian groups whose Pontryagin duals are not finitely generated,
		such as the $p$-adic integers $\Z_p$ for a prime number $p$.
		Hence,
		the inclusions $C(G^{(0)})\subset C^*_r(G)$ and $D\subset C^*_r(G)$ are not equivalent.
		\qed
	\end{proof}

	\begin{rem}
	In the previous corollary,
	we considered an \'etale groupoid $G$ such that $C^*_r(G)$ is a unital Kirchberg algebra.
	This is the case when $G$ is amenable, second countable minimal and locally contracting by \cite[Theorem 5.1]{Brown2014} and \cite[Proposition 2.4]{AnatharamanDelapurelyinfinite}.
	For example,
	the standard groupoid models of Cuntz algebras $\mathcal{O}_n$ satisfies these assumptions.
	More generally,
	these assumptions are satisfied by the standard groupoid models in \cite{Paterson2002} of graph algebras $C^*(E)$ for a finite directed graph $E$ such that every vertex connects to a loop and $E$ satisfies condition (L) by \cite[Theorem 3.9]{KumjianPaskRaeburnCuntzKrieger}.
	Note that these groupoid models are expansive since they can be realized as the transformation groupoids of finitely generated inverse semigroup actions with the properties in Example \ref{ex: finitely generated inverse semigroup actions yields expansive groupoids} by \cite{Paterson2002}.
	\end{rem}
	
	Next,
	we show Corollary \ref{cor: Gal group of prime subalgebra is fin gen},
	which is a variant of Theorem \ref{theorem every compact subgroup of Cartan fixing are fin gen}.
	While Theorem \ref{theorem every compact subgroup of Cartan fixing are fin gen}
	starts with a compact subgroup
	\[
	H\subset \Aut_{C_0(G^{(0)})}(C_r^*(G))
	\]
	with the prime fixed point subalgebra,
	Corollary \ref{cor: Gal group of prime subalgebra is fin gen} takes a prime intermediate subalgebra
	\[
	C_0(G^{(0)})\subset B\subset C_r^*(G)
	\]
	as its starting point and investigates the $B$-fixing automorphism group
	\[
	\Aut_B(C_r^*(G)).
	\]
	We prepare the following proposition to show Corollary \ref{cor: Gal group of prime subalgebra is fin gen}.
	
	\begin{prop}\label{prop: primeness of subalgebra passes to the ambient algebra}
		Let $G$ be a locally compact Hausdorff \'etale groupoid.
		Assume that $G$ is effective and an intermediate subalgebra $C_0(G^{(0)})\subset B\subset C^*_r(G)$ is prime.
		Then $C^*_r(G)$ is also prime and $G$ is topologically transitive.
	\end{prop}
	
	\begin{proof}
		Assume that $I\subset C^*_r(G)$ is a nonzero ideal.
		Since $I\cap C_0(G^{(0)})$ is an ideal of $C_0(G^{(0)})$,
		there exists an open set $U\subset G^{(0)}$ such that $I\cap C_0(G^{(0)})=C_0(U)$.
		Since $G$ is effective,
		$U$ is nonempty by Corollary \ref{cor: ideal intersection property}.
		By \cite[Lemma 3.1 (1)]{KangLiidealstr},
		$U$ is invariant.
		Note that $C^*_r(G_U)$ is an ideal of $C^*_r(G)$ by \cite[Lemma 3.3]{KangLiidealstr}.
		Put $J\defeq C^*_r(G_U)\cap B$.
		Since we have $C_0(U)\subset C^*_r(G_U)$ and $C_0(U)\subset C_0(G^{(0)})\subset B$,
		$J$ is a nonzero ideal of $B$ and hence an essential ideal of $B$.
		Suppose that $U$ is not dense in $G^{(0)}$.
		Then there exists a nonempty open set $V\subset G^{(0)}$ with $U\cap V=\emptyset$.
		Take $h\in C_0(V)\setminus\{0\}$.
		Then we have $hC^*_r(G_U)=\{0\}$ and therefore $hJ=\{0\}$.
		Since $J$ is essential,
		we obtain $h=0$,
		which is a contradiction.
		Thus $U$ is dense in $G^{(0)}$.
		
		Next,
		we observe that $C^*_r(G_U)$ is an essential ideal of $C^*_r(G)$.
		Assume that $a\in C^*_r(G)$ satisfies $aC^*_r(G_U)=\{0\}$.
		For any $f\in C_c(U)$,
		we have $af=0$ and hence
		\[
		E(a^*a)f=E(a^*af)=0,
		\]
		where $E\colon C^*_r(G)\to C_0(G^{(0)})$ denotes the standard conditional expectation.
		Thus we obtain $E(a^*a)|_U=0$.
		Since $U\subset G^{(0)}$ is dense,
		$E(a^*a)=0$.
		Since $E$ is faithful,
		we obtain $a=0$.
		Hence $C^*_r(G_U)$ is an essential ideal of $C^*_r(G)$.
		Since we have $C^*_r(G_U)\subset I$, 
		$I$ is also essential in $C^*_r(G)$.
		Therefore $C^*_r(G)$ is prime.
		Now,
		the last assertion follows from Corollary \ref{cor: prime is equivalent to topologically transitive}.
		\qed
	\end{proof}

	\begin{cor}\label{cor: Gal group of prime subalgebra is fin gen}
		Let $G$ be a locally compact Hausdorff expansive effective \'etale groupoid.
		Assume that an intermediate subalgebra $C_0(G^{(0)})\subset B\subset C^*_r(G)$ is prime and $H\defeq \Aut_B(C^*_r(G))$ is compact.
		Then $\widehat{H}$ is finitely generated.
	\end{cor}
	
	\begin{proof}
		Let $\tau\colon H\curvearrowright C^*_r(G)$ denote the canonical action of $H$.
		By \cite[Proposition 3.3.4]{Komura2025Weyl},
		there exists a continuous 1-cocycle $c\colon G\to\widehat{H}$ such that
		\[
		\tau_h(a)(\alpha)=c(\alpha)(h)a(\alpha)
		\]
		for all $h\in H$, $a\in C^*_r(G)$ and $\alpha\in G$.
		By \cite[Proposition 3.3.2]{Komura2025Weyl},
		we have $C^*_r(G)^{\tau}=C^*_r(\ker c)$.
		Since we assume that $B$ is prime and 
		\[B\subset C^*_r(G)^{\tau}=C^*_r(\ker c),\]
		$C^*_r(G)^{\tau}=C^*_r(\ker c)$ is also prime by Proposition \ref{prop: primeness of subalgebra passes to the ambient algebra}.
		Hence we may apply Theorem \ref{theorem every compact subgroup of Cartan fixing are fin gen} and $\widehat{H}$ is finitely generated.
		\qed
	\end{proof}

	\subsection{Compact abelian subgroups of $Z(G)$}
	
	In Theorem \ref{theorem every compact subgroup of Cartan fixing are fin gen},
	we observed that there is an obstruction to realizing compact abelian groups as subgroups of  $\Aut_{C_0(G^{(0)})}(C^*_r(G))$.
	In this subsection,
	we shift our focus to \'etale groupoids themselves rather than groupoid C*-algebras.
	Our aim in this subsection is to prove Proposition \ref{prop: Hperp is top transitive and G is expansive, Hhat is fin gen},
	which is a groupoid-theoretic analogue of Theorem \ref{theorem every compact subgroup of Cartan fixing are fin gen}.
	
	To this end,
	we begin by investigating groups of 1-cocycles on \'etale groupoids.
	Note that $\Hom(H,Z(G))$ and $Z(G,\widehat{H})$ in the following proposition are topological groups equipped with the compact-open topology.
	
	\begin{prop}\label{prop: Hom(H,Z(G)) is Z(G, hatH)}
		Let $G$ be a locally compact Hausdorff \'etale groupoid and $H$ be a locally compact abelian group.
		Then the following groups
		\begin{align*}
			\Hom(H, Z(G))&\defeq \{\sigma \colon H\to Z(G)\mid \text{$\sigma$ is a continuous group homomorphism}\},\\
			Z(G,\widehat{H})&\defeq \{c\colon G\to \widehat{H}\mid \text{$c$ is a continuous 1-cocycle}\}
		\end{align*}
		are isomorphic as topological groups.
		More precisely,
		for $\sigma\in\Hom(H,Z(G))$ and $c\in Z(G,\widehat{H})$,
		define $\widetilde{\sigma}\in Z(G,\widehat{H})$ and $\widetilde{c}\in \Hom(H,Z(G))$ by
		\begin{align*}
			\widetilde{\sigma}(\alpha)(h)&\defeq \sigma(h)(\alpha)\\
			\widetilde{c}(h)(\alpha)&\defeq c(\alpha)(h)
		\end{align*}
		for $\alpha\in G$ and $h\in H$.
		Then the map $\sigma\mapsto \widetilde{\sigma}$ defines an isomorphism from $\Hom(H,Z(G))$ to $Z(G,\widehat{H})$ and $c\mapsto \widetilde{c}$ is its inverse.
		
	\end{prop}
	
	\begin{proof}
		First,
		observe that $\widetilde{\sigma}\colon G\to \widehat{H}$ and $\widetilde{c}\colon H\to Z(G)$ are homomorphisms for $\sigma\in\Hom(H,Z(G))$ and $c\in Z(G,\widehat{H})$.
		Indeed,
		for $(\alpha,\beta)\in G^{(2)}$ and $h\in H$,
		we have
		\[
		\widetilde{\sigma}(\alpha\beta)(h)=\sigma(h)(\alpha\beta)=\sigma(h)(\alpha)\sigma(h)(\beta)=\widetilde{\sigma}(\alpha)(h)\widetilde{\sigma}(\beta)(h)
		\]
		and hence $\widetilde{\sigma}\colon G\to\widehat{H}$ is a homomorphism.
		One can see that $\widetilde{c}\colon H\to Z(G)$ is also a homomorphism in the same way.
		
		For topological spaces $X$ and $Y$,
		let $C(X,Y)$ denote the set of continuous maps from $X$ to $Y$,
		equipped with the compact-open topology.	
		Then,
		since we assume that $G$ and $H$ are locally compact Hausdorff,
		we have the natural identification
		\[
		C(H, C(G,\T))\simeq C(G, C(H,\T))(\simeq C(H\times G,\T))
		\]
		by \cite[Proposition A.14, A.16]{Hatcher_alg_top}.
		Note that this identification is a natural extension of the map
		\[
		\Hom(H,Z(G))\ni \sigma\mapsto\widetilde{\sigma}\in Z(G,\widehat{H}).
		\]
		Hence $\Hom(H,Z(G))$ and $Z(G,\widehat{H})$ are homeomorphic via this map.

		Now,
		it suffices to show that the map $\sigma\mapsto \widetilde{\sigma}$ is a group homomorphism from $\Hom(H,Z(G))$ to $Z(G,\widehat{H})$.
		Take $\sigma_1,\sigma_2\in\Hom (H, Z(G))$.
		Then,
		for $\alpha\in G$ and $h\in H$,
		\[
		\widetilde{\sigma_1\sigma_2}(\alpha)(h)=\sigma_1\sigma_2(h)(\alpha)=\sigma_1(h)(\alpha)\sigma_2(h)(\alpha)=\widetilde{\sigma_1}\widetilde{\sigma_2}(\alpha)(h).
		\]
		Hence $\sigma\mapsto \widetilde{\sigma}$ is a group homomorphism.
		This completes the proof.
		\qed
	\end{proof}

	For $A\subset Z(G)$,
	we define
	\[
	A^{\perp}\defeq \{\alpha\in G\mid \text{$\sigma(\alpha)=1$ for all $\sigma\in A$}\}.
	\]
	\begin{prop}\label{prop: ImHperp=kerhatsigma}
		Let $G$ be a locally compact Hausdorff \'etale groupoid and $H$ be a compact abelian group.
		For $\sigma\in\Hom(H,Z(G))$,
		\[
		\sigma(H)^{\perp}=\ker\widetilde{\sigma}.
		\]
		In particular,
		$\sigma(H)^{\perp}\subset G$ is a closed and open subgroupoid of $G$ with $G^{(0)}\subset \sigma(H)^{\perp}$.
	\end{prop}
	\begin{proof}
		The condition $\alpha\in \sigma(H)^{\perp}$ holds if and only if
		\[
		1=\sigma(h)(\alpha)=\widetilde{\sigma}(\alpha)(h)
		\]
		holds for all $h\in H$.
		Hence $\alpha\in \sigma(H)^{\perp}$ if and only if $\alpha\in\ker\widetilde{\sigma}$.
		This completes the proof of the former assertion.
		Now,
		the last assertion follows from the fact that $\widetilde{\sigma}\colon G\to \widehat{H}$ is continuous and $\widehat{H}$ is discrete. 
		\qed
	\end{proof}
	
	The following proposition is a variant of \cite[Proposition 3.3.3]{Komura2025Weyl}.
	In \cite[Proposition 3.3.3]{Komura2025Weyl},
	for $\sigma \in \Hom(H,Z(G))$,
	the author provided a sufficient condition for $\widetilde{\sigma}$ to be
	surjective, formulated in terms of the induced action $\widehat{\sigma}\colon H \curvearrowright C_r^*(G)$.
	In the next proposition, we give a sufficient condition for $\widetilde{\sigma}\colon G \to \widehat{H}$ to be surjective without using groupoid C*-algebras.

	\begin{prop}\label{prop: injectiveness of sigma is surjectivenss of hat sigma}
		Let $G$ be a locally compact Hausdorff \'etale groupoid,
		$H$ be a compact abelian group and $\sigma\in\Hom(H,Z(G))$.
		Then $\sigma\colon H\to Z(G)$ is injective if $\widetilde{\sigma}\colon G\to \widehat{H}$ is surjective.
		Conversely,
		if $\sigma\colon H\to Z(G)$ is injective and $\sigma(H)^{\perp}\subset G$ is topologically transitive,
		then $\widetilde{\sigma}\colon G\to \widehat{H}$ is surjective.
	\end{prop}
	
	\begin{proof}
		First, assume that $\widetilde{\sigma}$ is surjective and $h\in\ker\sigma$.
		Then we have
		\[
		\widetilde{\sigma}(\alpha)(h)=\sigma(h)(\alpha)=1
		\]
		for all $\alpha\in G$.
		Since we have $\widetilde{\sigma}(G)=\widehat{H}$,
		we obtain $h=e$.
		Next,
		assume that $\sigma$ is injective and $\sigma(H)^{\perp}$ is topologically transitive.
		Since we have $\ker\widetilde{\sigma}=\sigma(H)^{\perp}$ by Proposition \ref{prop: ImHperp=kerhatsigma},
		$\ker\widetilde{\sigma}\subset G$ is topologically transitive.
		Hence $\widetilde{\sigma}(G)\subset \widehat{H}$ is a subgroup by \cite[Lemma 2.2.2]{Komurasubmod}.
		If $\widetilde{\sigma}(G)\subsetneq \widehat{H}$,
		there exists $h\in H\setminus\{e\}$ such that $\widetilde{\sigma}(\alpha)(h)=1$ for all $\alpha\in G$.
		Since $\sigma(h)(\alpha)=\widetilde{\sigma}(\alpha)(h)=1$ holds for all $\alpha\in G$,
		we have $\sigma(h)=1$.
		Since $\sigma$ is injective,
		it follows $h=e$ and this contradicts to $h\in H\setminus\{e\}$.
		Hence we obtain $\widetilde{\sigma}(G)=\widehat{H}$ and $\widetilde{\sigma}$ is surjective.
		\qed
	\end{proof}

	The following proposition is the groupoid-theoretic analogue of Theorem \ref{theorem every compact subgroup of Cartan fixing are fin gen}.
	Compared to Theorem \ref{theorem every compact subgroup of Cartan fixing are fin gen},
	remark that we do not assume that $G$ is effective. 
		
	\begin{prop}\label{prop: Hperp is top transitive and G is expansive, Hhat is fin gen}
		Let $G$ be a locally compact Hausdorff expansive \'etale groupoid.
		Assume that a subgroup $H\subset Z(G)$ is compact and the subgroupoid
		\[
		H^{\perp}=\{\alpha\in G\mid \text{$h(\alpha)=1$ for all $h\in H$}\}
		\]
		of $G$ is topologically transitive.
		Then the Pontryagin dual group $\widehat{H}$ is finitely generated.
	\end{prop}
	
	\begin{proof}
		We denote the inclusion map by $\iota\colon H\to Z(G)$.
		By Proposition \ref{prop: Hom(H,Z(G)) is Z(G, hatH)},
		we obtain $\widetilde{\iota}\in Z(G,\widehat{H})$.
		Since we assume that $\ker \widetilde{\iota}=H^{\perp}$ is topologically transitive and $\iota$ is injective,
		$\widetilde{\iota}$ is surjective by Proposition \ref{prop: injectiveness of sigma is surjectivenss of hat sigma}.
		Now,
		$\widehat{H}$ is finitely generated by Proposition \ref{prop: image of expansive groupoid is finitely generated}.
		\qed
	\end{proof}

\bibliographystyle{plain}
\bibliography{bunken}

\end{document}